\documentclass[letter,11pt]{amsart}

\usepackage{hyperref}
\usepackage{amsmath,amsfonts}
\usepackage{amsthm}
\usepackage{amssymb}
\usepackage{url}
\def\R{\mathbb{R}}

\def\B{\mathbb{B}}

\usepackage{multirow}
\newcommand{\dsum}{\displaystyle\sum}
\usepackage{adjustbox}
\usepackage{natbib}
\usepackage{natbib}

\newtheorem{theorem}{Theorem}
\newtheorem{cor}[theorem]{Corollary}
\newtheorem{lem}[theorem]{Lemma}

\newtheorem{example}[theorem]{Example}
\newtheorem{rmk}[theorem]{Remark}

\let\origmaketitle\maketitle
\def\maketitle{
  \begingroup
  \def\uppercasenonmath##1{} 
  \let\MakeUppercase\relax 
  \origmaketitle
  \endgroup
}
\usepackage[ruled,vlined]{algorithm2e}
\usepackage{tikz}
\usetikzlibrary{calc,3d}

\usepackage{pgfplots}
\pgfplotsset{compat=newest}
\usetikzlibrary{decorations.markings}

\usepackage{caption}
\usepackage{subcaption}

\begin{document}

\title[Location of Leak Detection Devices]{\Large Optimal coverage-based placement of static leak detection devices for pipeline water supply networks}

\author[V. Blanco \MakeLowercase{and} M. Mart\'inez-Ant\'on]{{\large V\'ictor Blanco$^{a,b}$ and  Miguel Mart\'inez-Antón$^{a,b}$}\medskip\\
$^a$Institute of Mathematics (IMAG), Universidad de Granada\\
$^b$ Dpt. Quant. Methods for Economics \& Business, Universidad de Granada
}

\address{IMAG, Universidad de Granada, SPAIN.}
\email{vblanco@ugr.es}
\address{IMAG, Universidad de Granada, SPAIN.}
\email{mmanton@ugr.es}

\date{\today}

\begin{abstract}
In this paper, we provide a mathematical optimization-based framework to determine the location of leak detection devices along a network. Assuming that the devices are endowed with a known coverage area, we analyze two different models. The first model aims to minimize the number of devices to be located in order to (fully or partially) cover the volume of the network. In the second model, the number of devices is given, and the goal is to locate them to provide a coverage volume as broad as possible. Unlike other approaches in the literature, in our models, it is not assumed that the devices are located on the network (nodes or edges) but in the whole  space and that the different segments in the networks may be partially covered, which allows for more flexible coverage. We also derive a method to construct initial solutions as well as a math-heuristic approach for solving the problem for larger instances. We report the results of a series of experiments on real-world water supply pipeline networks, supporting the validity of our models.
\end{abstract}

\keywords{Facility Location, Leak Detection, Coverage Problems, Mixed Integer Non Linear Programming, Water Supply Networks.}

\maketitle

\section{Introduction}\label{sec:1}

The design of leak detection systems on water supply networks has attracted great interest due to the economic and environmental impact associated with the systematic loss of this resource.  Needless to say, the important role water plays in our social and economic life system, such as in agriculture, manufacturing, the production of electricity, and sustaining human health.

In urban networks, where the supply pipelines network is buried, an average of 20\% to 30\% of the supply water is periodically lost~\citep{el2019leak}. This average exceeds 50\% in places with less technological development where poor maintenance makes the system more vulnerable. It is also known~\citep{el2019leak} that $\sim$70\% of the amount of wasted water is provoked by losses caused by leaks in modern networks. Pipe internal roughness or friction factors are the main causes of leakage of a water pipeline network~\citep{Task1987,el2014condition}, and as the pipelines get older, they become more susceptible to damage. In developed countries, it is expected that the annual disbursements for water leaks in their supply networks would be close to 10 billion USD, of which 2 billion would go to costs for damages due to water loss and 8 billion to social effects costs.  Additionally, the International Institute of Water Management forecasts that 33\% of the world's population will experience water scarcity by 2025~\citep{seckler1998world}. Thus, the efficient management of water supplies is and will be one of the main concerns of water authorities throughout the world.

Most of the efforts related to the management of water supply networks have focused on the detection of leaks once they occur. Rapid detection of the leak location is then crucial to minimize the impact of the leaks.  \cite{hamilton2009alc} suggests three different  phases in the  leak detection problem: \emph{localization}, \emph{location} and \emph{pinpointing}. In the \emph{localization} phase, the goal is to detect if a leak has occurred within a certain network segment after a suspected leak. There are several proposed Machine Learning based methodologies to estimate leak probabilities or to classify the event leak/no leak  based on historic leakage datasets~\citep{el2016locating,li2011development}. In the \emph{location} phase, the uncertain area where the leak is localized is narrowed to $\sim 30$cm. Finally, in the \emph{pinpointing} phase, the exact position of the leak is to be determined with a pre-specified accuracy of $\sim 20$ cm by using hydrophones and/or geophones~\citep{fantozzi2009experience,royal2011site}. Previously to the determination of the position of the leak, a vast amount of literature have being devoted to modeling the determination of false/true leak alarms by the different available devices~\citep{cody2020linear,cody2020detecting}.

Another line of research on this topic is the design of control devices and methods for the accurate and quick detection of leakages. This is the case of the design of devices that accurately detect the leak within a restricted area~\citep{khulief2012acoustic}. Nevertheless, these devices are expensive and the adequate placement of the limited units must be strategically determined. One of the most popular approaches is by partitioning the network in \emph{district metered areas} where the flow and the pressure are monitored (leaks can be detected  by a decrease of flow and pressure) by means of leak-detection devices at each of this areas~\citep[see e.g.][]{puust2010review}. However, one still has to decide the number of devices and their exact locations in each of the district-metered areas.

There are different types of devices designed to help in the different leak detection phases which are classified into static and dynamic devices. Static devices, such as sensors or data loggers, are usually located on the network, at utility holes, or directly on the ground, attached to the network. They keep a data transmission flow with a central server to detect and localize leaks. In contrast, dynamic devices are portable and used in the location and pinpointing phases on more specific areas where the leak was suspected to occur. Whereas static devices can be automated, dynamic ones must be controlled on-site by humans. Different technologies have been designed for the two different types of devices~\cite[see e.g.][for further details]{li2015review}.

Most of the research on static leak detection systems is focused on the adequate estimation of the signals transmitted from the devices to the central server to detect an actual leak~\citep{mohamed2012leak,tijani2022improving}. A few works analyze the optimal placement of a given number of static devices on a finite number of potential placements based on the capability of each of the potential places to detect a leak~\citep{venkateswaran2018impact}, or in the use of historic data to place the devices at the more convenient places~\citep{casillas2013optimal}. 

This paper provides a technological decision support tool to help in the design of leak detection systems via the optimal placement of static devices. Instead of assuming that the devices are to be placed in a finite set of pre-specified potential places, they are allowed to be located in the whole space where the network lives, i.e. in the whole town, city, or district. We analyze, in this framework, two different strategies to place the devices. On the one hand, we derive a method to find the smallest number of devices (and their placements) required to detect \textit{any} leak in the whole network or in a given percent of it. Since the devices may be costly, covering a large amount of the network might be expensive, and we also derive a method, that fixes the number of devices to be located based on a budget, and finds their optimal placements to cover as much volume of the network as possible.

The models that we propose belong to the family of Continuous Covering Location Problems. In this type of problem, the goal is to find the position of one or more services (in this case, the leak detection devices), each of them endowed with a \emph{coverage area},  i.e., a limited region where the service/signal  can be provided. Covering Location Problems are usually classified into (Partial) Set Covering Location Problems--(P)SCLP and Maximal Coverage Location Problems--MCLP. The goal of the (P)SCLP is to determine the minimum number of services (or equivalently the minimum set-up cost for them) to cover (part of) a given demand. In MCLP the number of services is known and the goal is to place them to cover as much demand as possible. These problems have been widely studied in the literature in case the given demand points to cover are finite and planar, and the coverage areas are Euclidean disks~\cite[see][for further information on this problems]{garcia2015covering,murray2007coverage}. Several extensions of these problems have been studied, by imposing connectivity between the services in higher dimensional spaces and different coverage areas~\citep{blanco2021continuous}, multiple types of services~\citep{blanco2021multitype}, under uncertainty~\citep{hosseininezhad2013continuous}, regional demand~\citep{blanquero2016maximal}, or with ellipsoidal coverage areas~\citep{tedeschi2021new}. 

We provide versions of the PSCLP and the MCLP, where instead of covering demand points, the goal is to cover lengths/volumes of the water supply pipeline networks, and the services to be located are the devices to detect leaks. The goal is either to find the number of devices and their optimal placement to fully or partially cover the whole length of the network (in the case of the PSCLP) or to find the placements of a given number of devices to maximize the length of the network which is covered by the devices. We assume that the coverage areas of the devices are $\ell_\tau$-norm based balls and that covering a part of the network with these shapes implies that the device is able to detect a leak there. As far as we know, this problem has never been investigated before despite its practical applications. \cite{murray2007coverage} analyze planar covering problems with generalized types of demand, as line segments or polygons, but where partial coverage is not allowed. In our approach, a part of the different elements taking part of the network is allowed to be covered, being the overall coverage maximized or lower bounded in our models.

We derive mathematical programming formulations for our model. First, we analyze the simple case when a single device is to be located. Next, we extend the model to the case of the simultaneous location of more than one device. We propose Mixed Integer Non-Linear Programming formulations for the problems, that are reformulated as Mixed Integer Second-Order Cone Optimization problems. We analyze some properties of the model that allow us to develop a strategy to construct initial solutions by solving an Integer Linear Optimization problem. We also design a math-heuristic approach to approximately the problem by solving, sequentially, the single-device versions of the problem. We have tested all our approaches in real-world urban water networks. In addition to  analyzing the computational performance of our algorithms, we provide managerial insights about the locations obtained with our approaches, compared to the application of the classical algorithms in the literature, namely node and edge-restricted covering problems.

The rest of the paper is organized as follows. In section \ref{sec:2} we introduce the problem under analysis and illustrate some of the solutions that can be obtained. Section \ref{sec:3} is devoted to analyzing the problem of locating a single device, which will be helpful in the development of approximation algorithms for the multi-device case. In section \ref{sec:4} the general case is analyzed. We provide Mixed Integer Non-Linear Programming formulations for the maximal and partial set covering location problems and a deep study of them. We also provide a method to construct initial solutions for the problem based on the geometrical properties of the solutions, and a math-heuristic approach based on solving, iteratively, single-device instances. The results of our computational experiments on real-world urban pipeline networks are reported in Section \ref{sec:5}. Finally, in Section \ref{sec:6} we draw some conclusions and future research lines on the topic.

\section{Length-coverage location of devices}\label{sec:2}

In this section, we introduce the problem under study and fix the notation for the rest of the sections.

Let $G=(V,E;\Omega)$ be an undirected network with a set of nodes $V$, a set of edges $E$, and non-negative edge weights $\Omega$. The graph represents an urban water pipeline network, where the weights are the diameter or roughness of each of the pipelines in the network, which together with its length will allow us to compute the covered volume of the network. We assume that the graph is embedded in $\R^d$, i.e., $V \subseteq \R^d$  and each (undirected) edge $e=\{o_e, f_e\}\in E$ is  identified with a segment in $\R^d$, with endnodes $o_e$ and $f_e$ in $V$. Abusing notation, we identify edge $e\in E$ either with the segment induced by its end nodes, i.e., $e \equiv [o_e, f_e]$ or with the vector of $\R^d$ associated with them, i.e., $e\equiv f_e-o_e$.

A device located at  $X \in \R^d$ is endowed with a ball-shaped coverage area in the form:
$$
\B_R(X) = \{z \in \R^d: \|X-z\|\leq R\}
$$
where $R>0$ is the given coverage radius. We assume that $\|\cdot\|$ is an $\ell_\tau$-based norm with $\tau\geq 1$ or a polyhedral norm. Note that each of the devices can be endowed with a different radius and a different norm, based on their technical specifications.

For each edge $e\in E$, and a finite set of positions for the devices $\mathcal{X} \subset \R^d$, we denote by ${\rm CovWLength}_G(e,\mathcal{X})$ the weighted length of the edge covered by the devices. Let us denote by ${\rm TotWLength_G}$ the total weighted length of the network, i.e., ${\rm TotWLength_G} = \dsum_{e\in E} \omega_e\|o_e-f_e\|$ with $\omega_e\in \Omega$. 

We analyze in this paper two covering location problems  to determine the position of the leak detection devices, namely the Partial Set Network Length Covering Location Problem (PSNLCLP) and the Maximal Network Length Covering Location Problem  (MNLCLP). In both cases, the goal is to  place the different types of devices in the space to accurately detect a leak on the network. 

\begin{description}
\item[{\bf Partial Set Network Length Covering Location Problem (PSNLCLP)}]

The goal of this problem is to determine the minimum number of devices and their positions in $\R^d$ in order to cover at least $100\gamma\%$ of the weighted length of the network, for a given $\gamma \in (0,1]$. 

The PSNLCLP can be mathematically stated as:
$$
\min_{\mathcal{X} \subseteq \R^d:\atop {\sum_{e\in E} {\rm CovWLength}_G(e,\mathcal{X}) \geq \gamma {\rm TotWLength_G}}} |\mathcal{X}|
$$

The number of devices in the objective function can be replaced by the overall set-up costs for them, in whose case, the model read:
$$
\min_{\mathcal{X} \subseteq \R^d:\atop {\sum_{e\in E} {\rm CovWLength}_G(e,\mathcal{X}) \geq \gamma {\rm TotWLength_G}}} \dsum_{X \in \mathcal{X}} f_X
$$
being $f_X$ a given set-up cost for the device $X \in \mathcal{X}$.

For the sake of simplicity, in this paper, we analyze the first model, although all the results are also valid for the second one.

\item[{\bf Maximal Network Length Covering Location Problem (MNLCLP)}]

In this problem the number of devices to locate is given, $p\geq 1$, and the goal is to find their positions to maximize the weighted covered length of the network. 
While the MNLCLP consists of solving
$$
\max_{\mathcal{X} \subseteq \R^d:\atop |\mathcal{X}|=p} \sum_{e\in E} {\rm CovWLength}_G(e,\mathcal{X})
$$

\end{description}

Both in the PSNLCLP and the MNLCLP, one can provide different coverage for the different devices that want to be located. In the PSNLCLP, it is assumed that the number of available devices is unlimited (although minimized), but one can assume that different types of devices with different specifications are available. In the MNLCLP, since exactly $p$ of them are to be located, different radii can be specified for each of them.

In the following example, we illustrate the two problems described above analyzed in a real network (see Section \ref{sec:5}).
\begin{example}\label{ex:1}
Let us consider the network drawn in Figure \ref{fig:jilin0}. There, each edge has a different weight indicating the diameter of the pipeline (as larger the weight thicker the line in the plot). Devices with identical Euclidean disk coverage areas of radius $0.5$ are to be located (the network has been scaled to fit in a disk of radius $5$). In Figure \ref{fig:jilin} we show the solutions of the PSNLCLP for $\gamma=0.75$ (right) and the solution of MNLCLP for $p=5$ (left). There, the centers are highlighted as red stars, the covered segments of the network are colored in blue, and the coverage of the devices are the red disks.

Note that the flexible approach that we propose does not force the devices to be located at the nodes or edges of the network, being able to cover a larger amount of the volume of the network with a smaller number of devices.

\begin{figure}
\centering
\includegraphics[width=0.6\linewidth]{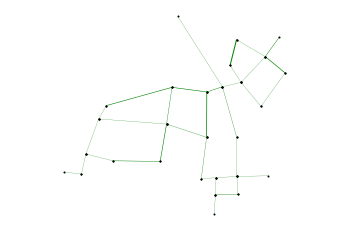}
\caption{Pipeline urban network of Example \ref{ex:1}.\label{fig:jilin0}}
\end{figure}

\begin{figure}
\centering
\includegraphics[width=0.45\linewidth]{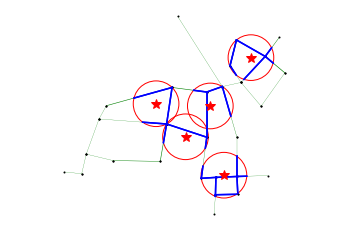}~
\includegraphics[width=0.45\linewidth]{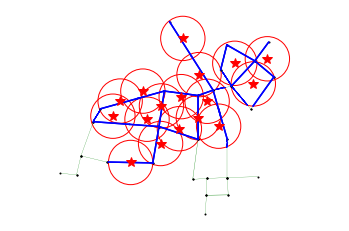}
\caption{Solutions of MNLCLP ($p=5$) and PSNLCLP ($\gamma=0.75$) of the network of Example \ref{ex:1}.\label{fig:jilin}}
\end{figure}

\end{example}

\begin{example}
The two problems that we introduce here are defined in a very general framework ($d$-dimensional spaces, networks with no further assumptions, and general coverage shapes). In Figure \ref{fig:l1linf} we show solutions for the MNLCLP for the same instance that in Example \ref{ex:1}, with $p=5$ but in case the coverage areas are induced by $\ell_1$-norm (left) and $\ell_\infty$-norm (right) balls.  

\begin{figure}[h!]
\includegraphics[width=0.45\linewidth]{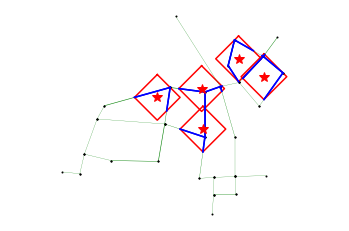}~
\includegraphics[width=0.45\linewidth]{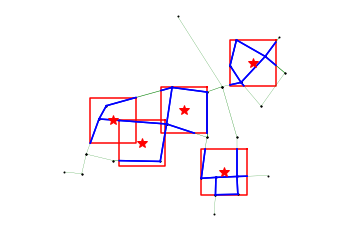}
\caption{Solutions of MNLCLP with $p=5$ for coverage areas defined by $\ell_1$-norm (left) and $\ell_\infty$-norm (right) balls.\label{fig:l1linf}}
\end{figure}
\end{example}

\begin{rmk}
As already mentioned, most covering location problems on networks assume that the centers must be located either on the edges or the nodes of the network~\cite[see e.g.]{berman2011minmax,berman2016covering}. Here, this condition is no longer assumed, allowing the centers to be located at any place in the space where the network lives. This flexibility allows positions for the devices providing a larger coverage of the network. In Figure \ref{fig:jilin_1} we show the solutions of the edge-restricted (left) and node-restricted (right) versions of the MNLCLP, where one can observe that the optimal positions of the devices are different from those obtained for the MNLCLP.

\begin{figure}[h!]
\includegraphics[width=0.45\linewidth]{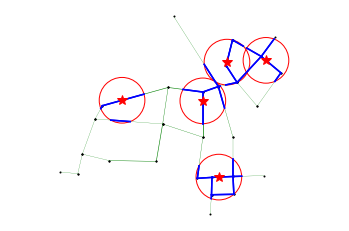}~
\includegraphics[width=0.45\linewidth]{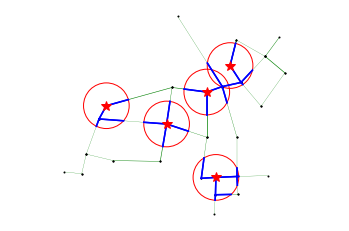}
\caption{Solutions of the edge-restricted and node-restricted versions of MNLCLP for $p=5$ for the network of Example \ref{ex:1}.\label{fig:jilin_1}}
\end{figure}

We have compared the covered lengths of the three problems (MNLCLP, edge-restricted MNLCLP, and node-restricted MNLCLP) for different values of $p$ ($2$, $5$, and $8$), and different radii $R$ ($0.1$, $0.25$, and $0.5$).  In Figure \ref{fig:edgesnodes} we show a bar diagram with the average deviations (for each $p$) of the two restricted versions with respect to the covered length of the general approach that we propose. As can be observed, the solutions of the unrestricted MNLCLP are able to cover more than $6\%$ than the edge-restricted problem and more than $20\%$ than the node-restricted problem. Since undetected leaks may produce fatal consequences in an urban area and leak detection devices are expensive, the use of the solutions of our models is advisable in this situation.
\begin{figure}
\centering
\includegraphics[width=0.45\linewidth]{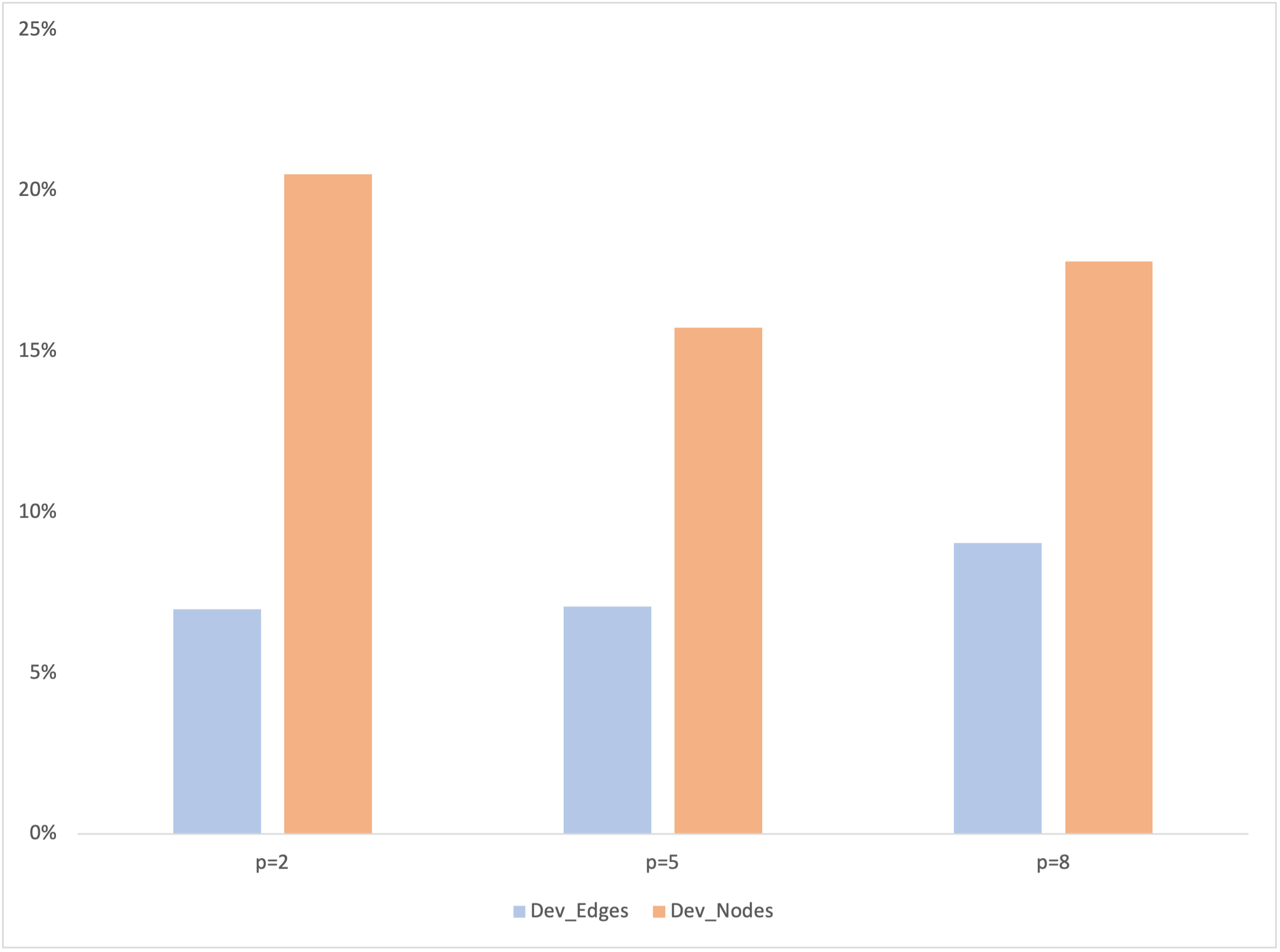}
\caption{Average length coverage deviations between the solutions of MNLCLP and the edges/nodes-restricted versions of the problem.\label{fig:edgesnodes}}
\end{figure}
\end{rmk}

\section{The single-device Maximal Network Length Covering Location Problem}\label{sec:3}

In this section, we first analyze the MNLCLP in case $p=1$ (a single device). We provide a mathematical programming model for the problem that will be useful for the general construction of the multi-device instances of MNLCLP and PSNLCLP derived in this paper.

Let $e\in E$ be an edge in the network and $X\in \R^d$ a given location for a device. In case the coverage area of the device in $X$, $\B_R(X)$, does not \emph{touch} the edge, then the covered length is  clearly zero. Otherwise, since $\B_R(X)$ is a compact and convex body in $\R^d$, $\partial \B_R(X)$, the border of the ball, will touch the segment in two points (that may coincide in case the segment belong to a tangent hyperplane of the ball). These points belong to the segment $[o_e, f_e]$, that can be parameterized as:
$$
Y_e^0 = \lambda_e^0 o_e + (1-\lambda_e^0) f_e \text{ and } Y_e^1 =\lambda_e^1 o_e + (1-\lambda_e^1) f_e 
$$
for some $\lambda_e^0, \lambda_e^1 \in [0,1]$. We can assume without loss of generality that $Y_e^0$ is closer to $o_e$ than $Y^1_e$, so we restrict the $\lambda$-values to  $\lambda_e^0 \leq \lambda_e^1$. 

With the above parameterization, the length of the edge covered by $X$ is $(\lambda_e^1 - \lambda_e^0) L_e$ (here, $L_e$ denotes the length of the edge $e$). 

To derive our mathematical programming formulation for the problem, we use the following sets of decision variables:

$$
z_ e = \begin{cases}
1 & \mbox{ if edge $e$ intersects the device's coverage area,}\\
0 & \mbox{otherwise}
\end{cases}
$$
$$
X: \text{ Coordinates of the placement of the device.}
$$
$$
Y_e^0, Y_e^1: \text{Intersections points of }\partial \B_R(X) \text{ with the edge } e
$$
$$
\lambda_e^0, \lambda_e^1: \text{ Parameterization values in the segment of intersection points }Y_e^0 \text{ and } Y_e^1, \text{ respectively.}
$$

The single-device MNLCLP can be formulated as the following Mathematical Programming Model, that we denote as ($1$-MNLCLP):

\begin{align}
\max & \;\;\; \dsum_{e\in E} \omega_e L_e (\lambda_e^1 - \lambda_e^0) \label{m1:0}\\
\mbox{s.t. } & \|X-Y_e^s\|z_e \leq R, \forall e\in E, s \in \{0,1\}, \label{m1:1}\\
 & Y_e^s =  \lambda_e^s o_e + (1-\lambda_e^s) f_e, \forall e \in E, s\in \{0,1\}, \label{m1:2}\\
 & \lambda_e^0 \leq \lambda_e^1, \forall e\in E, \label{m1:3}\\
  & \lambda_e^1 \leq z_e, \forall e\in E, s\in \{0,1\}, \label{m1:4}\\
 & \lambda_e^0, \lambda_e^1 \geq 0, \forall e\in E, s\in \{0,1\}, \label{m1:5}\\
 & z_e \in \{0,1\}, \forall e\in E, \label{m1:6}\\
 & X \in \R^d. \label{m1:7}
\end{align}

Constraints \ref{m1:1} enforce that in case the device coverage area intersects the edge, the intersection points must be in the coverage area of $X$. This constraint can be  equivalently rewritten as:
$$
\|X-Y_e^s\| \leq R + \Delta (1-z_e), \forall e\in E, s \in \{0,1\}
$$
where $\Delta$  a big enough constant with $\Delta> \max \Big\{\| z_1 -z_2\|: z_1, z_2 \in \{o_e, f_e: e\in E\}\Big\}$. Constraints \eqref{m1:2} are the parameterizations of the intersection points. Constraints \eqref{m1:3} force that $Y_e^0$ is closer to $o_e$ than $Y_e^1$. In case the device does not intersect an edge, by Constraints \eqref{m1:4} fix to zero the coefficients of the parameterization, adding a value of zero to the covered lengths in the objective function. \eqref{m1:5}-\eqref{m1:6} are the domains of the variables.

($1$-MNLCLP) is a Mixed Integer Non-Linear Programming problem because of the discrete variables $z$ and the nonlinear Constraints \eqref{m1:1}. For $\ell_\tau$ or polyhedral norms, these constraints are known to be efficiently rewritten as a set
 of second-order cone constraints (and in the case of polyhedral norms, as linear constraints) becoming a Mixed Integer Second-Order Cone Optimization (MISOCO) problem that can be solved using the off-the-shelf software \cite[see][for further details]{blanco2014revisiting}.

 \subsection{Generating feasible solutions of MNLCLP}\label{sec:3a}

The single-device version of the MNLCLP is already a challenging  combinatorial problem since it  requires computing a \textit{feasible} group of edges which is able to be covered by the device (in addition to the computation of the covered volume). In what follows, we analyze some geometrical properties and algorithmic strategies for this problem, that will result in an Integer Linear Programming formulation to generate good quality initial feasible solutions to this problem. The same ideas will be extended to generate solutions also for the multi-device problem.

The following result, whose proof is straightforward from Constraints \eqref{m1:1} provides a geometrical characterization of the potential position for the device to be located given that the \emph{touched} set of edges is known.
 \begin{lem}
Let $\bar z \in \{0,1\}^{|E|}$ be a feasible solution for $1$-MNLCLP Denote by $C=\{e \in E: \bar z_{e}=1\}$, the edges (total or partially) covered by the device. Then, we get that
\begin{equation}
X \in \bigcap_{e \in C} (e \oplus \B_{R}(0)),\label{cov}\tag{${\rm Cov}$}
\end{equation}
where $\oplus$ stands for the Minkowski sum in $\R^d$.
\end{lem}

The above result states that the position of the device, $X$, must belong to the intersection of the extended segments induced by the edges in the cluster $C$. In Figure \ref{caps} (left picture) we illustrate the shape of $e \oplus \B_{R}(0)$ for a given edge $e\in E$. In Figure \ref{caps} (right picture) we show the intersection of three of these types of sets, where a device covering the three segments is allowed to be located.

\begin{figure}
\begin{center}\begin{tabular}{ccc}
\includegraphics[width=0.3\textwidth, angle=85]{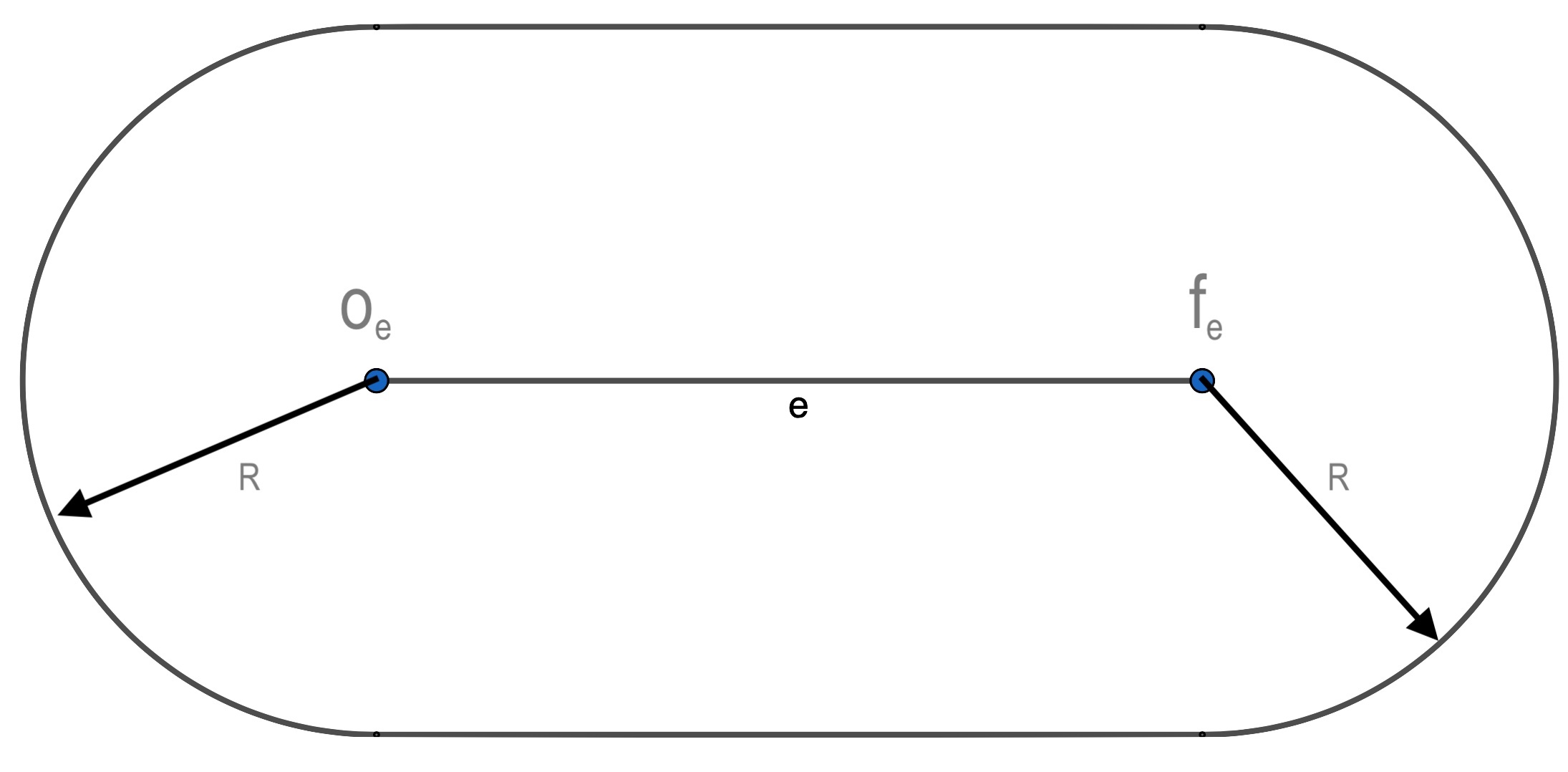} & &\includegraphics[width=0.4\textwidth]{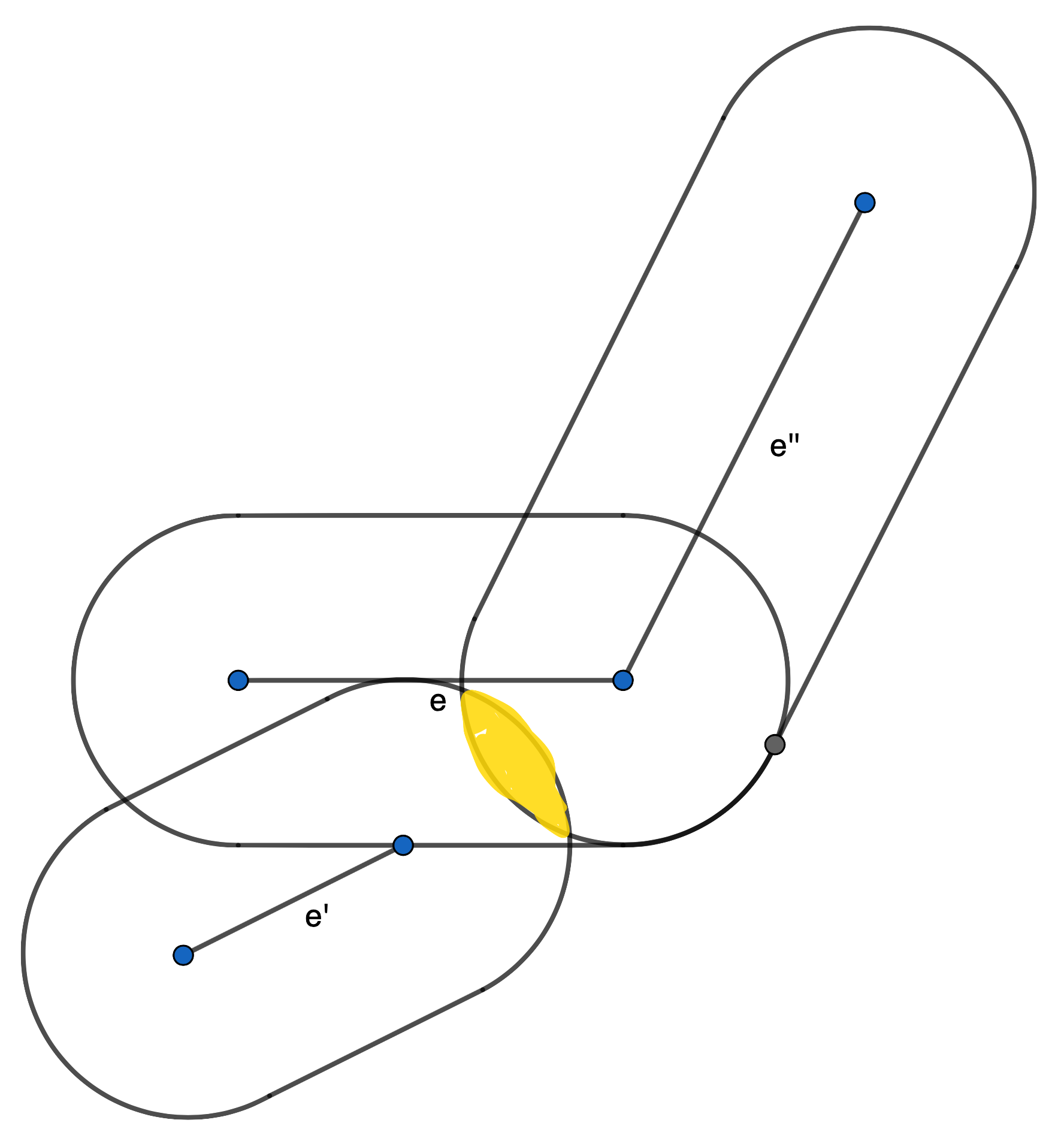}
\end{tabular}
\end{center}
\caption{Shape of extended edges (left) and the intersection of three of these \emph{compatible} shapes (right).\label{caps}}
\end{figure}

As already mentioned, the main combinatorial decision of our models is to determine the sets of edges that are allowed to be \emph{touched} by the same device, i.e., $S \subset E$, such that $\displaystyle\bigcap_{e \in S} (e \oplus \B_{R}(0))\neq \emptyset.$ Defining the set $\mathcal{C}$ as
$$
\mathcal{C} = \left\{S \subset E: \bigcap_{e \in S} (e \oplus \B_{R}(0)) \neq \emptyset\right\}
$$
each element $S \in \mathcal{C}$ will be called a \emph{compatible subset} for the device. In general, unless the radius is big enough, not all the subsets of $E$ belong to $\mathcal{C}$. 

In the following result, we describe a polynomial set (in $|E|$) of valid inequalities for our model that  avoid those non-compatible sets in our models.
\begin{lem}
The following inequalities are valid for the $1$-MNLCLP:
\begin{equation}\label{helly}
\dsum_{e\in S} z_{e} \leq d, \forall S\subset E \text{ with $|S|=d+1$ and } \bigcap_{e \in S} (e \oplus \B_{R}(0)) = \emptyset
\end{equation}
\end{lem}
\begin{proof}
It is straightforward to see, by Constraints \eqref{m1:1}, that the following condition is verified by any solution of \eqref{m1:0}-\eqref{m1:7}:
$$
\dsum_{e\in S} z_{e} \leq |S|-1, \forall S\subset E: \bigcap_{e \in S} (e \oplus \B_{R}(0)) = \emptyset, 
$$
Thus, by Helly's theorem~\citep{Helly:1923}, since the sets taking part of the intersections above, $(e \oplus \B_{R}(0))$, are compact and convex for any $e\in E$, the result follows.
\end{proof}

\begin{cor}
Let $\bar z \in \{0,1\}^{|E|}$ be a solution of the system of Diophantine inequations \eqref{helly}. Then, $\bar z$ is a feasible solution for  the $1$-MNLCLP. 
\end{cor}

In the classical Maximal Coverage Location Problems, the above observation allows us to replace the non-linear covering constraints --in the shape of \eqref{m1:1}-- with inequalities in the shape of \eqref{helly}, and the continuous variables which are involved in these constraints can be forgotten (see \cite[e.g.]{blanco2021continuous,blanco2021multitype}). In our model, the latter is no longer possible as in classic MCLP since the $\lambda$-values are also needed to compute the covered volume of the network. 

Thus, we propose the following Integer Linear Programming formulation to obtain valid compatible subsets for the models.

\begin{align}
\max &\dsum_{e\in E} \omega_e L_e z_{e}\label{m_0}\\
\mbox{s.t. } & \dsum_{e\in S} z_{e} \leq d, \forall S\subset E (|S|=d+1): \bigcap_{e \in S} (e \oplus \B_{R}(0)) = \emptyset,\label{m_exp}\\
& z_{e}\in \{0,1\}, \forall e\in E.\label{m_f}
\end{align}

The above mathematical programming model is an edge-based  version of the classical $1$-Maximal Coverage Location Problem, which is known to be NP-hard. The main advantage of this representation is that one can use techniques from Integer Linear Programming to strengthen or solve it, using the available off-the-shelf solvers.

The main  bottleneck of this formulation is the computation of  the intersections of $d+1$ sets in the form $e\oplus\B_R(0)$ which are empty, in whose case the corresponding inequality is added  to the pool
 of constraints. The general methodology that can be applied for any dimension and any $\ell_\tau$-based norm, is by applying a relax-and-cut approach based on solving the problems above by removing Constraints \eqref{m_exp}, separating the violated constraints and incorporating them on-the-fly in an embedded branch-and-cut algorithm. 
 
 In what follows we focus on the planar Euclidean case, which is the most useful case in practice, and for which the formulations can be further simplified.

 Observe that for $d=2$, Constraints \eqref{helly} are equivalent to:
 \begin{align*}
 z_{e} + z_{e^\prime} \leq 1, &\forall e, e^\prime\in E:  (e \oplus \B_{R}(0)) \cap   (e^\prime \oplus \B_{R}(0)) =\emptyset,\\
z_{e} + z_{e^\prime} +  z_{e^{\prime\prime}}\leq 2, &\forall e, e^\prime\in E:  (e \oplus \B_{R}(0)) \cap   (e^\prime \oplus \B_{R}(0)) \cap   (e^{\prime\prime} \oplus \B_{R}(0)) =\emptyset,\\
z_{e}\in \{0,1\}, &\forall e\in E.
\end{align*}
Thus, in order to incorporate these types of constraints one needs to check two and three-wise intersections of objects in the form $e \oplus \B_{R}(0)$. Although these shapes can be difficult to handle in general, the planar Euclidean case can be efficiently handled by analyzing the geometry of these objects as Minkowski sums of segments and disks.
 
 The following results are instrumental for the development of the algorithm that we propose to generate the above sets of constraints. From now on, $\|\cdot\|$ denotes the Euclidean norm in $\R^2$.
 \begin{lem}\label{lema:par}
 	Let $e,e'$ be two segments in $\R^2$ and $\delta(e,e')=\min\{\|X-X'\|: X \in e, X' \in e'\}$. Then, if $\delta(e,e')>0$, there exist $X\in e$ and $X'\in e'$ with $\delta(e,e') = \|X-X'\|$ such that either $X \in \{o_e,f_e\}$ or $X' \in \{o_{e'},f_{e'}\}$.
 \end{lem}

\begin{proof}
The result follows by observing that the minimum distance between two segments is always achieved by choosing one of the extremes of the segments.
\end{proof}

\begin{lem}\label{lema:punto-seg}
	Let $e$ be a segment in $\R^2$, and $Q\in \R^2$. Then, $\delta(e,Q) := \min \{\|Q-X\|: X \in e\}$ can be computed as:
	$$
	\delta(e,Q) = \|Q -  (\min\{\max\{0,\mu\}, 1\} (f_e-o_e) + o_e)\|.
	$$
\end{lem}
\begin{proof}
    The constructive proof is detailed in \ref{appendix}.
\end{proof}

Given a set of edges, $E$, and a radius $R$, using the above results, we develop  algorithms to compute the two and three-wise intersections of sets in the form $e \oplus \B_{R}(0)$, for $e\in E$. The pseudocodes are shown in Algorithms \ref{alg:1} and \ref{alg:2}. The set $M$, which is initialized to the empty set, will contain, the pairs $(e,e')$ of $E\times E$ with $(e\oplus \B_R(0))\cap (e'\oplus \B_R(0))=\emptyset$ by checking the distance between the segments, $\delta(e,e')$. On the one hand, in case, $\delta(e,e')=0$, both segments intersect so also their Minkowski sums. On the other hand, if $\delta(e,e')\neq 0$, we denote by $r_e$ and $r_{e'}$ the lines containing the segments $e$ and $e'$, respectively, and by $Q_0$ their intersection point. By Lemma \ref{lema:par} there exist $X, X'\in \R^2$ with $\delta(e,e') = \|X-X'\|$, being $X \in \{o_e, f_e\}$ or $X' \in \{o_{e'}, f_{e'}\}$. Thus, four distances are enough to compute $\delta(e,e')$, namely $\delta_1 = \delta(o_{e'}, e)$, $\delta_2 = \delta(f_{e'}, e)$, $\delta_3 = \delta(o_{e}, e')$ and $\delta_4 = \delta(f_{e}, e')$, being $\delta(e,e') = \min \{\delta_1, \delta_2, \delta_3, \delta_4\}$. In case $\delta(e,e') > 2R$, then $(e\oplus \B_R(0))\cap (e'\oplus \B_R(0)) = \emptyset$, and the tuple $(e,e')$ is added to $M$.

For the three-wise intersections, the set $M$ is again initialized to the empty set. Then, for each triplet $(e_1,e_2,e_3)$ whose pairwise intersection is non-empty (by Algorithm \ref{alg:1}), we solve the following mathematical optimization problem:
\begin{align}
\varepsilon^*(e_1, e_2, e_3) := \min & \ \varepsilon \label{eq0:alg2}\\
\mbox{s.t. } &Y_i=(1-\lambda_i)o_{e_i} + \lambda_i f_{e_i}, i=1,2,3, \\
& \| X-Y_i\|\leq R +\varepsilon, i=1, 2, 3,\\
& X \in \R^2,\\
& \lambda_1, \lambda_2, \lambda_3 \in [0,1],\\
& \varepsilon \in \R.\label{eqf:alg2}
\end{align}

The above problem is polynomial-time solvable since it can be rewritten as a continuous Second-Order Cone Optimization problem. Furthermore, its objective value provides a way to check for the emptiness of the three-wise intersection problem.
\begin{lem}\label{3wise}
Let $e_1, e_2, e_3 \in E$ with $(e_i\oplus \B_R(0)) \cap (e_j\oplus \B_R(0)) \neq \emptyset$ for all $i, j \in \{1,2,3\}$. Then, $(e_1\oplus \B_R(0)) \cap (e_2\oplus \B_R(0)) \cap (e_3\oplus \B_R(0)) \neq \emptyset$ if and only if $\varepsilon^*(e_1,e_2,e_3)=0$.
\end{lem}


 \begin{algorithm}[h!]
 \KwData{Set of edges, $E$, and radius $R$.}

 
 $M=\emptyset$
 
 \For{$(e,e') \in E \times E$}{
    Set: $\bar{e}=f_e-o_e$.\\
    Set: $\bar{e'}=f_{e'}-o_{e'}$.\\
 	Compute the intersection point of  the lines $o_e + \langle \bar{e} \rangle$ and $o_{e'} + \langle \bar{e'} \rangle$: $Q_0$.\\
 	Calculate $\mu_0$, $\mu_0'$ such that $Q_0=\mu_0\bar{e}+o_e$ and $Q_0=\mu_0'\bar{e'}+o_{e'}$.\\
 	\If{$\mu_0$ or $\mu_0'\notin [0,1]$}{
 		\begin{enumerate}
 			\item Compute the intersection point of  the lines $o_e + \langle \bar{e} \rangle$ and $o_{e'} + \langle \bar{e}^\bot \rangle$: $Q_1$.\\
 			Calculate $\mu_1$ such that $Q_1 = \mu_1 \bar{e} + o_e$.\\
 			Set: $\delta_1 = \|o_{e'} -  (\min\{\max\{0,\mu_1\}, 1\} \bar{e} + o_e)\|$.
 			\item Compute the intersection point of  the lines $o_e + \langle \bar{e} \rangle$ and $f_{e'} + \langle \bar{e}^\bot \rangle$: $Q_2$.\\
  			Calculate $\mu_2$ such that $Q_2 = \mu_2 \bar{e} + o_e$.\\
  			Set: $\delta_2 = \|f_{e'} -  (\min\{\max\{0,\mu_2\}, 1\} \bar{e} +  o_e)\|$.
   			\item Compute the intersection point of  the lines $o_e + \langle \bar{e'}^\bot \rangle$ and $o_{e'} + \langle \bar{e'} \rangle$: $Q_3$.\\
  			Calculate $\mu_3$ such that $Q_3= \mu_3 \bar{e'} + o_{e'}$.\\
  			Set: $\delta_3 = \|o_e -  (\min\{\max\{0,\mu_3\}, 1\} \bar{e'} +  o_{e'})\|$.
   			\item Compute the intersection point of  the lines $f_{e} + \langle \bar{e'}^{\bot} \rangle$ and $o_{e'} + \langle \bar{e'} \rangle$: $Q_4$.\\
  			Calculate $\mu_4$ such that $Q_4= \mu_4 \bar{e'} + o_{e'}$.\\
  			Set: $\delta_4 = \|f_e -  (\min\{\max\{0,\mu_4\}, 1\} \bar{e'} +  o_{e'})\|$.
  		\end{enumerate}
  		\If{$\min\{\delta_1, \delta_2, \delta_3, \delta_4\} > 2R$}{
  		Add $(e,e')$ to $M$.
    	}
	}
 }

 \KwResult{$M = \left\{ (e,e')\in E\times E : (e \oplus \B_{R}(0)) \cap   (e^\prime \oplus \B_{R}(0)) = \emptyset \right\}$.}
 
 \caption{A complete set of $2$-wise incompatible edges.\label{alg:1}}
\end{algorithm}
\begin{algorithm}[h!]
 \KwData{Set of edges, $E$, and radius $R$.}

 $L=\{(e_1,e_2,e_3) \in E\times E\times E: (e_1 \oplus \B_{R}(0)) \cap   (e_2 \oplus \B_{R}(0))\neq\emptyset, (e \oplus \B_{R}(0)) \cap   (e_3 \oplus \B_{R}(0))\neq\emptyset, (e_2 \oplus \B_{R}(0)) \cap   (e_3 \oplus \B_{R}(0))\neq\emptyset\}$.
 
 $M_3=\emptyset$.
 
 \For{$(e_1,e_2,e_3) \in L$}{
 	Compute $\varepsilon^*(e_1,e_2,e_3)$. 

		  \If{$\varepsilon^*(e_1,e_2,e_3) > 0$}{
  Add $(e_1, e_2, e_3)$ to $M_3$.
  }
  
 }

 \KwResult{$M = \left\{ (e_1,e_2,e_3)\in E\times E\times E : (e_1 \oplus \B_{R}(0)) \cap   (e_2 \oplus \B_{R}(0))\cap (e_3 \oplus \B_{R}(0)) = \emptyset \right\}$.}
 
 \caption{A complete set of $3$-wise incompatible edges (which are pair-wise compatible).\label{alg:2}}
\end{algorithm}

 \section{A general model for (PSNLCLP) and (MNLCLP)}\label{sec:4}
 
 In this section, we provide a general methodology to deal with the optimal location of devices in both the PSNLCLP and the MNLCLP. In the single-device problem analyzed in the previous section, the coverage of an edge can be directly computed by parameterizing the intersection of the boundary of the ball with the edge. Nevertheless, in the multi-device problem, the covered length does not coincide with the sum of the coverages of every single device separately, since the same part of a segment may be covered by two or more devices, but the covered length must be accounted for only once (otherwise the optimal placement for a set of devices is the collocation of all of them in the more weighted edge). 
 
 We illustrate the situation in the following toy example.
 \begin{example}
 Let us consider a planar network with a single edge $e$ and four devices with Euclidean ball coverage areas as drawn in Figure \ref{fig:sev}. The four devices \emph{touch} the edge. The covered length of the edge is highlighted with thicker segments in the picture. Clearly, this length cannot be computed by adding up separately each of the covered lengths of the devices.
 
\begin{figure}[h!]
\centering\begin{tikzpicture}[scale=0.5]
\tikzstyle{arrow} = [thick,scale=3,->,>=stealth]
\coordinate(O) at (2,2);
\coordinate(D) at (20,8);

\coordinate(X1) at (7,3);
\coordinate(X2) at (6,5);
\coordinate(X3) at (16,7);
\coordinate(X4) at (10,4);

\node[circle,draw,inner sep=1.5pt](D-1) at (D) {\tiny $f_e$};
\node[circle,draw,inner sep=1.5pt](O-1) at (O) {\tiny $o_e$};

\coordinate(Y11) at (5,3);
\coordinate(Y12) at (5.3381,3.1127);
\coordinate(Y21) at (7.6618,3.8872);
\coordinate(Y22) at (8, 4);
\coordinate(Y31) at (8.6, 4.2);
\coordinate(Y32) at (11.6, 5.2);
\coordinate(Y41) at (14.2265, 6.0755);
\coordinate(Y42) at (17.9734,7.3244);

\draw[color=red!60, fill=red!5, very thick, fill opacity=0.5](X1) circle (2);
\draw[color=blue!60, fill=blue!5, very thick, fill opacity=0.5](X2) circle (2);
\draw[color=green!60, fill=green!5, very thick, fill opacity=0.5](X3) circle (2);
\draw[color=yellow!60, fill=yellow!5, very thick, fill opacity=0.5](X4) circle (2);
	
\node[circle,draw,fill,inner sep=1pt, red](Y11-1) at (Y11) {};
\node[circle,draw,fill,inner sep=1pt, blue](Y12-1) at (Y12) {};
\node[circle,draw,fill,inner sep=1pt, blue](Y21-1) at (Y21) {};
\node[circle,draw,fill,inner sep=1pt, yellow](Y22-1) at (Y22) {};
\node[circle,draw,fill,inner sep=1pt, red ](Y31-1) at (Y31) {};
\node[circle,draw,fill,inner sep=1pt, yellow](Y32-1) at (Y32) {};
\node[circle,draw,fill,inner sep=1pt, green](Y41-1) at (Y41) {};
\node[circle,draw,fill,inner sep=1pt, green](Y42-1) at (Y42) {};
\draw (O-1) --(D-1);
\draw[very thick]  (Y11)--(Y32);
\draw[very thick]  (Y41)--(Y42);
			 
\end{tikzpicture}
\caption{Example of interaction between the coverages of different devices.\label{fig:sev}}
\end{figure}
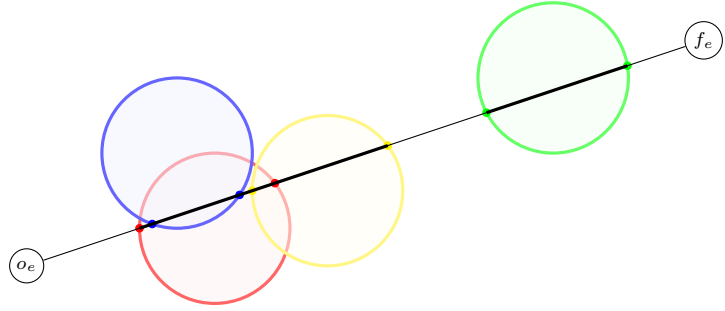
 \end{example}

In what follows we derive a mathematical programming model that overcomes this situation.

Observe that the positions of the intersection points of the coverage areas of $p$ devices with an edge (segment) $e$ provide a partition of the edge in at most $2p+1$ subsegments of $e$. Each of those subsegments is either fully covered or non-covered by the device. Let $\lambda^0_{1e},  \lambda^1_{1e}, \ldots, \lambda^0_{pe},  \lambda^1_{pe}$ the parameterizations of the intersection points of the $p$ devices with respect to $e$ (here $\lambda_{je}^0$ and $\lambda_{je}^1$ stands for the parameterizations of the intersection of the coverage area of $j$-th device with segment induced by the edge $e$). 

We assume that the devices not intersecting the edge have both lambda values equal to zero. Sorting the $\lambda^0$ and $\lambda^1$ values one gets two sorted sequences in the form:
$$
\Lambda^0_e := \lambda^0_{(1)e} \leq \cdots \leq \lambda^0_{(p)e}
$$
$$
\Lambda^1_e := \lambda^1_{(1)e} \leq \cdots \leq \lambda^1_{(p)e}
$$
Merging both lists one gets all the partitions of the segment $e$ by the different intersection points:
$$
\Lambda_e := \lambda^{i_1}_{(1)e} \leq \cdots \leq \lambda^{i_{2p}}_{(2p)e}
$$
where $i_1, \ldots, i_{2p} \in \{0,1\}$.

For each $l \in \{1, \ldots, 2p\}$, the intervals $[\lambda^{i_l}_{(l)e}, \lambda^{i_{l+1}}_{(l+1)e}]$ along with $[o_e,\lambda^{i_1}_{(1)e}]$ and $[\lambda^{i_{2p}}_{(2p)e}, f_e]$ induce a subdivision of the segment $e$ into $2p+1$ pieces (some of them probably singletons). Furthermore, in case any of the extreme points, $o_e$ and $f_e$, are covered by a device, the corresponding subsegment will be a singleton ($[o_e,o_e]=\{o_e\}$ or $[f_e,f_e]=\{f_e\}$), whose length is zero. Otherwise, in case any extreme subsegment is not covered, still its contribution to the objective function is zero. Thus, in our formulation, it is enough considering the $2p-1$ intermediate subsegments in the subdivision.

Given the sequence $\Lambda_e$ for the $p$ given devices located at $X_1, \ldots, X_p$, one can easily determine which of the subsegments in the partitions are covered by the facilities as stated by the following straightforward observation.

\begin{lem}\label{lemma1}
A subsegment in the form $s=[\lambda^{i_l}_{(l)e}, \lambda^{i_{l+1}}_{(l+1)e}]$ is covered by a set of devices if and only if $s \subseteq [\lambda_{je}^0, \lambda_{je}^1]$ for some $j=1,\dots,p$ with $\lambda_{je}^0 <\lambda_{je}^1$.
\end{lem}

With the above observations, we derive mathematical programming formulations for the multi-device versions of PSNLCLP and MNLCLP.

We denote by $P=\{1, \ldots, p\}$ the index set for the devices to locate and by $Q=\{1, \ldots, 2p-1\}$ the index sets for the intermediate  subsegments in the partition induced by the $\Lambda$ sequences. We assume that the $j$-th device is endowed with a $\|\cdot\|$-based coverage area with radius $R_j$.

We use the following decision variables in our models:

$$
z_ {je} = \begin{cases}
1 & \mbox{ if edge $e$ intersect the $j$-th device's coverage area,}\\
0 & \mbox{otherwise}
\end{cases} \forall j\in P, \forall e\in E.
$$
$$
X_{j1}, \ldots, X_{jd}: \text{ Coordinates of the $j$-th device, } \forall j \in P.
$$
\begin{align*}
\lambda_{je}^0, \lambda_{je}^1: & \text{ Parameterization in the segment of the two intersection}\\
& \text{ points of $\partial\B_{R_j}(X_j)$ with segment $e$, } \forall j\in P, \forall e\in E.
\end{align*}
$$
w_{e\ell} =\begin{cases}
1 & \mbox{if the $\ell$-th subsegment of edge $e$ is covered by some device,}\\
0 & \mbox{otherwise}
\end{cases}, \forall \ell\in Q, \forall e \in E.
$$

 $$
\xi_{je\ell}^s=\begin{cases}
1 & \text{if  $\lambda^s_{je}$ is sorted in $\ell$-th position in the list of $\Lambda_e$},\\
0 & \text{otherwise}
\end{cases} \forall j\in P, \forall \ell\in Q\cup \{2p\}, \forall e \in E.
$$

With the above set of variables, the amount:
$$
L_e\left[\dsum_{j\in P}\dsum_{s=0}^1\lambda_{je}^s\xi_{je(\ell+1)}^s-\dsum_{j\in P}\dsum_{s=0}^1\lambda_{je}^s\xi_{je\ell}^s\right]
$$
determines the length of the $\ell$-th subsegment in case it is covered by any of the devices in $P$. Note that in case such a subsegment is $[\lambda_{je}^s, \lambda_{j'e}^{s'}]$, the above expression becomes $L_e (\lambda_{j'e}^{s'}-\lambda_{je}^s)$ which is the desired amount. 

Thus, the overall volume coverage of the network can be computed as:
$$
\dsum_{e\in E}\dsum_{\ell\in Q} \omega_e w_{e\ell} L_e \left[\dsum_{j\in P}\dsum_{s=0}^1\lambda_{je}^s\xi_{je(\ell+1)}^s-\dsum_{j\in P}\dsum_{s=0}^1\lambda_{je}^s\xi_{je\ell}^s\right]
$$

In order to adequately represent the decision variables in our model, the following constraints are considered:

\begin{enumerate}
\item \emph{Coverage Constraints:}
\begin{equation}\label{2}
\|(\lambda_{je}^s e + o_e) - X_j\| z_{je} \leq R_j,   \forall  j\in P,  \ \forall e\in E,  \ s=0,1
\end{equation}
These constraints enforce that in case an edge is accounted as \emph{touched} by the $j$-th device ($z_{je}=1$),  the two intersection points $(\lambda_{je}^0 e + o_e)$ and $(\lambda_{je}^1 e + o_e)$ belong to $\B_{R_j}(X_j) \cap e$. This constraint can be reformulated as:
$$
\|(\lambda_{je}^s e + o_e) - X_j\| z_{je} \leq R_j + \Delta(1-z_{je}),   \forall  j\in P,  \ \forall e\in E,  \ s=0,1
$$
where $\Delta$  a big enough constant with $\Delta> \max \Big\{\| z_1 -z_2\|: z_1, z_2 \in \{o_e, f_e: e\in E\}\Big\}$.
\item \emph{Directed Parameterization:}
\begin{equation}
\lambda_{je}^0\leq \lambda_{je}^1,  \forall  j\in P, \ \forall e\in E. \label{3}
\end{equation}
In case the coverage area of a device $j$ touches the segment $e$, the segment is oriented in the parameterization.
\item \emph{Zero parameterizations for untouched edges}
\begin{equation}
\lambda_{je}^1\leq z_{je},  \forall  j\in P, \ \forall e\in E. \label{4}
\end{equation}
In case the $j$-th device does not touch the segment induced by an edge $e$, the covered length of such an edge by the device will be zero. By \eqref{2}, in that case, the device is not restricted to touching the segment, but to assure that no length is accounted for, we fix both $\lambda$-values in the fictitious intersection to zero.
\item \emph{$\Lambda$-Sorting Constraints:}
\begin{align}
\dsum_{j\in P} (\xi_{je\ell}^0 + \xi_{je\ell}^1) &=1,  \forall e\in E,  \ \forall \ell\in Q\cup \{2p\},\label{5}\\
\dsum_{\ell\in Q\cup \{2p\}} \xi_{je\ell}^s&=1,   \forall j\in P, \ \forall e\in E,  \ s=0,1 \label{6} \\
\dsum_{j\in P}(\lambda_{je}^0 \xi_{je\ell}^0 + \lambda_{je}^1 \xi_{je\ell}^1)&\leq \dsum_{j\in P}(\lambda_{je}^0 \xi_{je(\ell+1)}^0 + \lambda_{je}^1 \xi_{je(\ell+1)}^1), \forall  e\in E,  \ \forall \ell\in Q.\label{7}
\end{align}
These constraints allow us to adequately define the variables $\xi$. Constraints \eqref{5} and \eqref{6} assure that for each $e$ each $\lambda_e$-value is sorted in exactly a single position in $Q$ and that each position is assigned to exactly one $\lambda_e$ value. Constraint \eqref{7} enforces that the $\xi$-variables sort the $\lambda$-values in non-decreasing order.
\item \emph{Coverage of subsegments:}
\begin{align}
w_{e\ell}&\leq\dsum_{j\in P}\left(\dsum_{i\leq \ell}\xi_{jei}^0+\dsum_{i >\ell}\xi_{jei}^1-1\right),  \forall e\in E,  \ \forall \ell\in Q,\label{9}
\end{align}
The coverage of a subsegment $\ell\in Q$ is assured by the existence  of a device $j$ for which its $\lambda_{je}^0$ is sorted in a position anterior to $\ell$ ($\sum_{i\leq \ell}\xi_{jei}^0=1$) and $\lambda_{je}^1$ in a posterior position to $\ell$ ($\sum_{i>\ell}\xi_{jei}^1=1$). Thus, in case both values are $1$, the conditions of Lemma \ref{lemma1} are verified, and the subsegment is covered. Otherwise, one of the above sums is zero, and the constraint is redundant.  Indeed, if $\sum_{i\leq \ell}\xi_{jei}^0=0$, then, by \eqref{6}, $\sum_{i> \ell}\xi_{jei}^0=1$. Thus, by \eqref{3} and \eqref{7}, one has that  $\sum_{i>\ell}\xi_{jei}^1=1$. Similarly, if $\sum_{i>\ell}\xi_{jei}^1=0$, one has that $\sum_{i\leq \ell}\xi_{jei}^0=1$. In both cases, $\sum_{i\leq \ell}\xi_{jei}^0+\sum_{i >\ell}\xi_{jei}^1-1$ takes value zero, implying that the $j$-th device does not cover the subsegment.
\end{enumerate}

Apart from the constraints above, we incorporate into our model the following valid inequalities that allow us to strengthen it:

\begin{enumerate}
\item \emph{Touched segments and covered subsegments:} 
$$
\dsum_{\ell\in Q} w_{e\ell}\leq 2\dsum_{j\in P} z_{je}, \;\; \forall e\in E.
$$
In case the whole segment is not touched by any device, non of the subsegments are covered.
\item \emph{Symmetry breaking:}
$$
\dsum_{k=1}^d X_{jk} \leq \dsum_{k=1}^d X_{(j+1)k}, \;\; \forall j\in P, j<p.
$$
Since the devices to be located are indistinguishable, any permutation of the $j$-index will result in an alternative optimal solution, hindering the solution procedure based on a branch-and-bound tree. The above inequality prevents such an amount of alternative optima.
\item \emph{Incompatible edges:}
$$
z_{ej} + z_{e'j} \leq 1, \;\; \forall j \in P, \forall e, e'\in E \text{ with } \min\{\|x-x'\|: x \in e, x' \in e'\} >2R.
$$
Edges that are far enough are not able to be simultaneously touched by the same device. 
\end{enumerate}

\vspace*{1cm}

{\bf \underline{Mathematical Programming Model for (MNLCLP)}}:

\vspace*{1cm}

Using the variables and constraints previously described, the following mathematical programming formulation is valid for the MNLCLP:

\begin{align*}
\max & &\ \dsum_{e\in E}\dsum_{\ell\in Q} \omega_e w_{e\ell} L_e\left[\dsum_{j\in P}\dsum_{s=0}^1\lambda_{je}^s\xi_{je(\ell+1)}^s-\dsum_{j\in P}\dsum_{s=0}^1\lambda_{je}^s\xi_{je\ell}^s\right]
\end{align*}
\begin{align*}
\mbox{s.t. } 
& \eqref{2}-\eqref{9},\\
& \lambda_{je}^s \in [0,1], & \forall j\in P, \forall e \in E, s=0,1,\\
& X_j \in \R^d, &\forall j\in P,\\
& z_{je} \in \{0,1\},& \forall j\in P, \forall e \in E,\\
& \xi_{je\ell}^s \in \{0,1\}, & \forall j\in P, \forall e\in E, \forall \ell \in Q\cup \{2p\}, s=0,1,\\
& w_{e\ell} \in \{0,1\}, &\forall e\in E, \forall \ell \in Q.
\end{align*}

\vspace*{1cm}

{\bf \underline{Mathematical Programming Model for (PSNLCLP)}:}

\vspace*{1cm}

(PSNLCLP) seeks to minimize the number of devices to cover at least a portion $\gamma \in (0,1]$ of the length of the network. Although the above variables and constraints can be used to derive similarly a model for this problem, the number of devices, $p$, to locate is unknown in this case. We estimate an upper bound for this parameter and consider the following binary variables to \emph{activate}/\emph{desactivate} them.
$$
y_j = \begin{cases}
1 & \mbox{if device }j\mbox{ is activated,}\\
0 & \mbox{otherwise}.
\end{cases} \forall j \in P.
$$

Then, the (PSNLCLP) can be formulated as follows:
\begin{align}
\min & \dsum_{j\in P} y_j\nonumber\\
\mbox{s.t. } 
& \eqref{2}-\eqref{9},\\
& \dsum_{e\in E}\dsum_{\ell\in Q} \omega_e w_{e\ell} L_e\left[\dsum_{j\in P}\dsum_{s=0}^1\lambda_{je}^s\xi_{je(\ell+1)}^s-\dsum_{j\in P}\dsum_{s=0}^1\lambda_{je}^s\xi_{je\ell}^s\right] \geq \gamma \dsum_{e\in E} \omega_e L_e,  \label{gamma}\\
& z_{je} \leq y_j, \;\;\forall j \in P, \;\;\forall e \in E;\label{zy}\\
& \lambda_{je}^s \in [0,1],  \forall j\in P, \forall e \in E, s=0,1,\nonumber\\
& X_j \in \R^d, \forall j\in P,\nonumber\\
& z_{je} \in \{0,1\}, \forall j\in P, \forall e \in E,\nonumber\\
& \xi_{je\ell}^s \in \{0,1\}, \forall j\in P, \forall e\in E, \forall \ell \in Q\cup \{2p\}, s=0,1,\nonumber\\
& w_{e\ell} \in \{0,1\}, \forall e\in E, \forall \ell \in Q,\nonumber\\
& y_j \in \{0,1\}, \forall j \in P.\nonumber
\end{align}

In this case, the objective function accounts for the number of activated devices, Constraint \eqref{gamma} assures that at least a portion of $\gamma$ of the coverage volume is attained, and \eqref{zy} prevents covering edges by devices that are not activated. 

Instead of minimizing the number of devices one may also minimize the set-up costs by incorporating individual set-up costs for each of the available devices ($f_j $ for $j\in P$) and replace the above objective function by $\sum_{j\in P} f_j y_j$.

To avoid multiple optimal solutions due to symmetry, we also incorporate into the model the following constraints that avoid activating the $j$-th device in case the $(j-1)$-th device is not activated in the solution.
$$
y_{j-1} \geq y_{j}, \forall j \in P, j>1.
$$

\begin{rmk}
The complexity of the PSNLCLP highly depends on the number of potential devices to locate ($p$), since the number of constraints and variables are affected by this parameter. We derive a method to compute a reasonable upper bound for that parameter which is based on computing the minimum number of devices necessary to cover each edge in $U_\gamma \subseteq E$ where $U_\gamma$ is defined as a minimal set verifying that $\dsum_{e \in U_\gamma} \omega_e L_e \geq \gamma \dsum_{e\in E} \omega_e L_e$. 

We initialize $U_\gamma=\emptyset$, and sort the sequence $\left\{\omega_e L_e\right\}_{e\in E}$ such that $\omega_{e_1} L_{e_1}\geq \omega_{e_2} L_{e_2}\geq\cdots\geq \omega_{e_i} L_{e_i}\geq \omega_{e_{i+1}} L_{e_{i+1}}\geq \cdots$. Then, we define $U_\gamma = \{e_1, \ldots, e_k\}$ such that:
$$
\dsum_{i=1}^k \omega_{e_i} L_{e_i} \geq \gamma \dsum_{e\in E} \omega_e L_e > \dsum_{i=1}^{k-1} \omega_{e_i} L_{e_i}.
$$
Since the minimum number of devices necessary to cover a single edge $e$ is $\left\lceil\frac{L_e}{2R}\right\rceil$, we can fix $p=\dsum_{e\in U_\gamma}\left\lceil\frac{L_e}{2R}\right\rceil$.
\end{rmk}

The Mixed Integer Non-Linear Programming models that we develop for (MNLCLP) and (PSNLCLP) have $O(p^2 |E|)$ variables, $O(p|E|)$ linear constraints, and $O(p|E| f_{\|\cdot\|})$ nonlinear constraints (here, $f_{\|\cdot\|}$ stand for the number of constraints that allow rewriting Constraints \ref{2} as second-order cone constraints (see \citep{blanco2014revisiting} for upper bounds on this number for $\ell_\tau$-norms). Thus, it is advisable in these models to design alternative solution strategies for solving them or to provide initial solutions that alleviate the search for optimal solutions by providing lower bounds for our problem.  In the following sections, we propose different alternatives taking advantage of the geometric properties of these problems.

\subsection{Constructing initial feasible solutions}\label{sec:5a}

The geometric properties that we derive in Section \ref{sec:3a} for the single device problem can be also extended to the $p$-device case. Specifically, one can construct solutions of MNLCLP by avoiding the computation of covered lengths in the models and assuming that once an edge of the network is touched by the coverage area of a device, the whole is accounted as covered. With these assumptions, we construct initial solutions to our problem by solving the following integer linear programs:
\begin{align}
\max &\dsum_{e\in E} \dsum_{j\in P} \omega_e L_e z_{je}\label{m_0:p}\\
\mbox{s.t. } & \dsum_{j\in P} z_{je} \leq 1, \forall e \in E,\\
&\dsum_{e\in S} z_{je} \leq d, \forall S\subset E (|S|=d+1): \bigcap_{e \in S} (e \oplus \B_{R}(0)) = \emptyset, \forall j \in P,\label{m_exp:p}\\
& z_{je}\in \{0,1\}, \forall e\in E, \forall j \in P.\label{m_f:p}
\end{align}
In the problem above, the overall weighted length of the covered edges is to be maximized by restricting edges to be covered by the same device to those which are feasible for the MNLCLP. The edges are also enforced to be accounted for at most once in the solution.

The strategies for generating and separating the constraints of the above problem are identical to those detailed in Section \ref{sec:3a}. 

As can be observed in our computational experience (Section \ref{sec:5}), the incorporation of these initial solutions to the original formulation of MNLCLP is advisable to derive exact optimal solutions to the problem in less CPU time.

\subsection{Math-heuristic approach}\label{sec:5b}

In addition to the exact approaches provided by formulations MNLCLP or PSNLCLP and the generation of initial solutions in the previous section, we propose a math-heuristic procedure to obtain good quality approximated solutions for the problems for larger size instances, in less CPU time. The math-heuristic is based on solving, sequentially, the single-device location problem \eqref{m1:1}-\eqref{m1:7} that was described in Section \ref{sec:3a}. 

We show in Algorithm \ref{alg_math2} a pseudocode for this procedure. As already mentioned, the approach is based on solving, sequentially, a single-device location device problem until a certain termination criterion (which depends on the problem to solve, MNLCLP or PSNLCLP) is verified. In case the problem is the MNLCLP the algorithm ends when the number of devices in the pool reaches the value of $p$. Otherwise, for the PSNLCLP the algorithm ends when the covered weighted length reaches the desired value. 

At each iteration, a device is located, and the network to be covered in the next iteration is updated from the previous iteration by removing the segments that have been covered.

\begin{algorithm}[h!]
 \KwData{Network $G=(V,E;\Omega)$, number of devices $p$ and radius $R$.}

 $V'=V, E'=E, \Omega'=\Omega$\\
 $X=\emptyset$\\
 
 \While{\texttt{Termination\_Criterion}}{
    Solve $X', \lambda_e^0, \lambda_e^1, z_e=$ \text{arg} \eqref{m1:0}-\eqref{m1:7} for $e\in E', \omega_e\in \Omega'$ and $R$.\\
    Update \texttt{Termination\_Criterion}
    Add $X'$ to $X$.\\
 	\For{$e \in E'$}{
 	      \If{$z_e=1$}{
            \If{$\lambda_e^0\in (0,1)$}{
                Add $Y_e^0$ to $V'$.\\
                Add $\{o_e,Y_e^0\}$ to $E'$.\\
                Add $\omega_e$ to $\Omega'$.\\
                }
            \If{$\lambda_e^1\in (0,1)$}{
                Add $Y_e^1$ to $V'$.\\
                Add $\{Y_e^1,f_e\}$ to $E'$.\\
                Add $\omega_e$ to $\Omega'$.\\
                }
            Remove $e$ from $E'$
         }
         }
         }
\KwResult{$X\in \R^{d\times p}$: Location of the devices.}
 \caption{Math-heuristic 2.\label{alg_math2}}
\end{algorithm}

\section{Computational Experiments}\label{sec:5}

In this section, we report on the results of a series of computational experiments performed to empirically assess our methodological contribution to the p-MNLCLP and PSNLCLP presented in the previous sections. We use six real networks obtained from two different sources: one based on the networks developed by the University of Exeter's (UOE) Centre for Water Systems available in \url{https://emps.exeter.ac.uk/engineering/research/cws/resources/benchmarks/} and other privately provided by Prof. Ormsbee from the University of Kentucky (UKY). These networks, which are called \texttt{gessler}, \texttt{jilin}, \texttt{richmond}, \texttt{foss}, \texttt{rural} and \texttt{zj}, have $14, 34, 44, 58, 60$ and $85$ edges, respectively. The networks have been scaled to fit in a disk of radius $5$. The networks are drawn in Figure \ref{fig:networks}.

\begin{figure}[h!]
\begin{subfigure}[b]{0.3\textwidth}
\centering
\includegraphics[width=\linewidth]{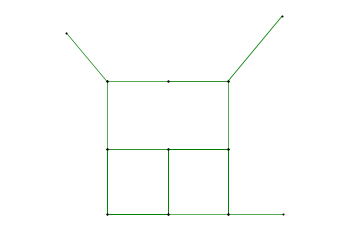}
\caption{\texttt{gessler}}
\end{subfigure}
\hspace{0.2cm}
\begin{subfigure}[b]{0.3\textwidth}
\centering
\includegraphics[width=\linewidth]{jilin.png}
\caption{\texttt{jilin}}
\end{subfigure}
\hspace{0.2cm}
\begin{subfigure}[b]{0.3\textwidth}
\centering
\includegraphics[width=\linewidth]{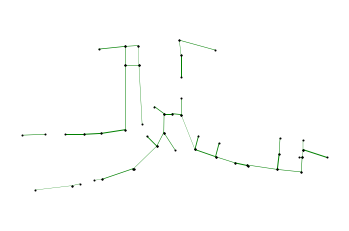}
\caption{\texttt{richmond}}
\end{subfigure}
     \begin{subfigure}[b]{0.3\textwidth}
         \centering
        \includegraphics[width=\linewidth]{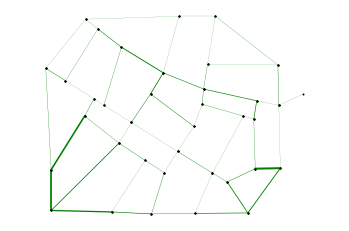}
        \caption{\texttt{foss}}
\end{subfigure}
\hspace{0.2cm}
     \begin{subfigure}[b]{0.3\textwidth}
         \centering
\centering
\includegraphics[width=\linewidth]{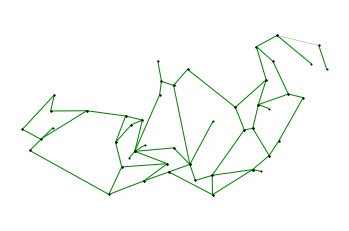}
\caption{\texttt{rural}}
\end{subfigure}
\hspace{0.2cm}
     \begin{subfigure}[b]{0.3\textwidth}
\centering
\includegraphics[width=\linewidth]{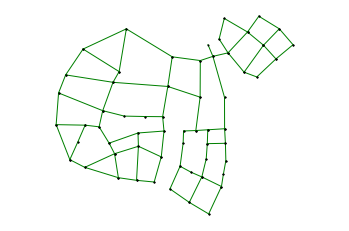}
\caption{\texttt{zj}}
\end{subfigure}
\caption{Networks used in our computational experiments. \label{fig:networks}}
\end{figure}

We have run the different approaches for the MNLCLP and the PSNLCLP for disk-shaped coverage areas with radii ranging in $\{0.1,0.25,0.5\}$. For the MNLCLP the number of devices to locate, $p$, ranges in $\{2,5,8\}$, whereas for the PSNLCLP the values of $\gamma$ range in $\{0.5, 0.75, 1\}$.

All the experiments have been run on a virtual machine in a physical server equipped
with 12 threads from a processor AMD EPYC 7402P 24-Core Processor, 64 Gb of RAM, and
running a 64-bit Linux operating system. The models were coded in Python 3.7 and we used
Gurobi 9.1 as optimization solver. A time limit of $5$ hours was set for all the experiments.

In Tables \ref{tab:mclp} and \ref{tab:pscp} we show the average results obtained in our experiments. We report average values of the consumed CPU time (in seconds), and percent of unsolved instances, and the MIP Gap within the time limit. Both tables are similarly organized. In the first block (first three columns), the name of the instance together with its number of nodes and edges is provided. In the second block (next two columns) we write the values of $p$ (for the MNLCLP) or $\gamma$ (for the PSNLCLP) and the radius. The next three blocks are the results obtained with each of the approaches. For the MNLCLP we run the MISOCO formulation, and also the exact and math-heuristic solution approaches detailed in section \ref{sec:5a} (MNLCLP\_1, for short) and \ref{sec:5b} (MNLCLP\_2), respectively.  We do not report results on the Unsolved instances and MIPGap for the MNLCLP\_2 since all the instances were solved within the time limit with that approach. In Table \ref{tab:pscp} the results are organized similarly for the PSNLCLP, but we do not generate initial solutions since that strategy only applies to the MNLCLP, and only the strategy PSNLCLP\_2. The flag \texttt{TL} indicates that all the instances averaged in the row reach the time limit without certifying optimality. The flag \texttt{OoM} indicates that the solver outputs \emph{Out of Memory} at some point when solving the instance. 

The first observation from the results that we obtain is that both problems are computationally challenging since they require large CPU times to solve even the small instances. Actually, the exact MNLCLP was only able to solve up to optimality, small instances with small values of $p$, and the exact PSNLCLP only solved a few instances, and in many of them, the solver outputs Out of Memory when solving them.

The first alternative (and exact) strategy, MNLCLP\_1, based on constructing initial solutions to the problem, had a slightly better performance with respect to those instances that were solved with the initial formulation, both in CPU time and MIPGap. Some of the instances were not able to be solved with MNLCLP but were able to be solved with the initial solutions that we construct. 

With respect to the heuristic approach, the consumed CPU times are tiny compared to the times required by the exact approaches and was able to construct feasible solutions for all the instances, even for those that the exact approaches flagged Out of Memory. In terms of the quality of the obtained solutions, in Figure \ref{fig:devs} we show the average deviations (for each instance) of the alternative approaches with respect to the original one. This measure provides the percent improvement of the alternative method with respect to the best solution obtained by the original formulation of the problem. We observed that the solutions that we obtained with the two strategies are significantly better than those obtained with the original formulation for the MNLCLP within the time limit. Providing initial solutions to the problem allows us to obtain solutions with 20\% more coverage than the initial formulation, whereas the heuristic approach gets solutions with more than $25\%$ more coverage. In the case of the PSNLCLP, in most, if instances the solutions of the heuristic are better than the ones obtained with the exact approach, but in instance \texttt{jilin}, the solutions are $20\%$ worse than the obtained with the exact approach.

In Figure \ref{fig:boxplot} we plot boxplots for the average of the best weighted coverages obtained by each of the approaches within the time limit for each of the six instances. Regarding the distribution of the results, one cannot infer that the results are significantly different. Nevertheless, in terms of the central tendencies, there is a clear improvement in the results obtained with the two alternative methods with respect to the original formulation, especially for the larger instances (see Figure \ref{fig:devs}).

\begin{table}[htbp]
  \centering
   \adjustbox{scale=0.8}{\begin{tabular}{|p{1.5cm}p{0.4cm}p{0.4cm}|p{0.2cm}r|rrr|rr|rr|}\hline
          &       &       &       &       & \multicolumn{3}{c|}{CPU Time (secs)} & \multicolumn{2}{c|}{Unsolved}       & \multicolumn{2}{c|}{GAP (\%)} \\
    {Instance} & {$|V|$} & {$|E|$} & {$p$} & {$R$} & MNLCLP  & MNLCLP\_1 & {MNLCLP\_2} & {MNLCLP} & {MNCLP\_1} & {MNLCLP} & {MNLCLP\_1} \\\hline\hline
    \multirow{9}[0]{*}{\texttt{gessler}} & \multirow{9}[0]{*}{12} & \multirow{9}[0]{*}{14} & \multirow{3}[0]{*}{2} & 0.1   & {151.53} & {13.69} & 0.89  & 0\%   & 0\%   & 0\%   & 0\% \\
          &       &       &       & 0.25  & {48.97} & {11.87} & 1.34  & 0\%   & 0\%   & 0\%   & 0\% \\
          &       &       &       & 0.5   & {26.28} & {10.59} & 0.62  & 0\%   & 0\%   & 0\%   & 0\% \\\cline{4-12}
          &       &       & \multirow{3}[0]{*}{5} & 0.1   & \texttt{TL} & \texttt{TL} & 2.26  & 100\% & 100\% & 86\%  & 84\% \\
          &       &       &       & 0.25  & \texttt{TL} & \texttt{TL} & 2.92     & 100\% & 100\% & 69\%  & 62\% \\
          &       &       &       & 0.5   & \texttt{TL} & \texttt{TL} & 1.61  & 100\% & 100\% & 24\%  & 31\% \\\cline{4-12}
          &       &       & \multirow{3}[0]{*}{8} & 0.1   & \texttt{TL} & \texttt{TL} & 3.54  & 100\% & 100\% & 90\%  & 87\% \\
          &       &       &       & 0.25  & \texttt{TL} & \texttt{TL} & 5.59  & 100\% & 100\% & 74\%  & 69\% \\
          &       &       &       & 0.5   & \texttt{TL} & \texttt{TL} & 2.92  & 100\% & 100\% & 41\%  & 35\% \\\hline
    \multirow{9}[0]{*}{\texttt{jilin}} & \multirow{9}[0]{*}{28} & \multirow{9}[0]{*}{34} & \multirow{3}[0]{*}{2} & 0.1   & {167.25} & {39.10} & 1.99  & 0\%   & 0\%   & 0\%   & 0\% \\
          &       &       &       & 0.25  & {196.56} & {144.30} & 3.37  & 0\%   & 0\%   & 0\%   & 0\% \\
          &       &       &       & 0.5   & {164.83} & {152.10} & 2.45  & 0\%   & 0\%   & 0\%   & 0\% \\\cline{4-12}
          &       &       & \multirow{3}[0]{*}{5} & 0.1   & \texttt{TL} & \texttt{TL} & 2.95  & 100\% & 100\% & 86\%  & 85\% \\
          &       &       &       & 0.25  & \texttt{TL} & \texttt{TL} & 6.64  & 100\% & 100\% & 72\%  & 64\% \\
          &       &       &       & 0.5   & \texttt{TL} & \texttt{TL} & 3.17  & 100\% & 100\% & 40\%  & 42\% \\\cline{4-12}
          &       &       & \multirow{3}[0]{*}{8} & 0.1   & \texttt{TL} & \texttt{TL} & 6.07  & 100\% & 100\% & 88\%  & 84\% \\
          &       &       &       & 0.25  & \texttt{TL} & \texttt{TL} & 10.34 & 100\% & 100\% & 72\%  & 73\% \\
          &       &       &       & 0.5   & \texttt{TL} & \texttt{TL} & 4.67  & 100\% & 100\% & 70\%  & 37\% \\\hline
    \multirow{9}[0]{*}{\texttt{richmond}} & \multirow{9}[0]{*}{48} & \multirow{9}[0]{*}{44} & \multirow{3}[0]{*}{2} & 0.1   & {1180.62} & {133.99} & 8.75  & 0\%   & 0\%   & 0\%   & 0\% \\
          &       &       &       & 0.25  & {717.09} & {121.90} & 7.47  & 0\%   & 0\%   & 0\%   & 0\% \\
          &       &       &       & 0.5   & {184.63} & {244.25} & 2.32  & 0\%   & 0\%   & 0\%   & 0\% \\\cline{4-12}
          &       &       & \multirow{3}[0]{*}{5} & 0.1   & \texttt{TL} & \texttt{TL} & 23.22 & 100\% & 100\% & 78\%  & 77\% \\
          &       &       &       & 0.25  & \texttt{TL} & \texttt{TL} & 13.79 & 100\% & 100\% & 62\%  & 59\% \\
          &       &       &       & 0.5   & \texttt{TL} & \texttt{TL} & 3.70  & 100\% & 100\% & 42\%  & 41\% \\\cline{4-12}
          &       &       & \multirow{3}[0]{*}{8} & 0.1   & \texttt{TL} & \texttt{TL} & 33.64 & 100\% & 100\% & 88\%  & 85\% \\
          &       &       &       & 0.25  & \texttt{TL} & \texttt{TL} & 23.89 & 100\% & 100\% & 86\%  & 71\% \\
          &       &       &       & 0.5   & \texttt{TL} & \texttt{TL} & 5.82  & 100\% & 100\% & 71\%  & 56\% \\\hline
    \multirow{9}[0]{*}{\texttt{foss}} & \multirow{9}[0]{*}{37} & \multirow{9}[0]{*}{58} & \multirow{3}[0]{*}{2} & 0.1   & {561.98} & {39.61} & 2.77  & 0\%   & 0\%   & 0\%   & 0\% \\
          &       &       &       & 0.25  & {380.54} & {38.42} & 1.99  & 0\%   & 0\%   & 0\%   & 0\% \\
          &       &       &       & 0.5   & {196.92} & {86.40} & 1.83  & 0\%   & 0\%   & 0\%   & 0\% \\\cline{4-12}
          &       &       & \multirow{3}[0]{*}{5} & 0.1   & \texttt{TL} & \texttt{TL} & 6.49  & 100\% & 100\% & 82\%  & 80\% \\
          &       &       &       & 0.25  & \texttt{TL} & \texttt{TL} & 5.46  & 100\% & 100\% & 64\%  & 62\% \\
          &       &       &       & 0.5   & \texttt{TL} & \texttt{TL} & 4.31  & 100\% & 100\% & 61\%  & 56\% \\\cline{4-12}
          &       &       & \multirow{3}[0]{*}{8} & 0.1   & \texttt{TL} & \texttt{TL} & 9.33  & 100\% & 100\% & 88\%  & 86\% \\
          &       &       &       & 0.25  & \texttt{TL} & \texttt{TL} & 7.99  & 100\% & 100\% & 87\%  & 71\% \\
          &       &       &       & 0.5   & \texttt{TL} & \texttt{TL} & 9.11  & 100\% & 100\% & 78\%  & 64\% \\\hline
    \multirow{9}[0]{*}{\texttt{rural}} & \multirow{9}[0]{*}{48} & \multirow{9}[0]{*}{60} & \multirow{3}[0]{*}{2} & 0.1   & {12263.72} & {1169.41} & 16.94 & 0\%   & 0\%   & 0\%   & 0\% \\
          &       &       &       & 0.25  & \texttt{TL} & {559.93} & 15.69 & 100\% & 0\%   & 23\%  & 0\% \\
          &       &       &       & 0.5   & {5054.64} & {1612.73} & 13.98 & 0\%   & 0\%   & 0\%   & 0\% \\\cline{4-12}
          &       &       & \multirow{3}[0]{*}{5} & 0.1   & \texttt{TL} & \texttt{TL} & 26.46 & 100\% & 100\% & 92\%  & 91\% \\
          &       &       &       & 0.25  & \texttt{TL} & \texttt{TL} & 32.19 & 100\% & 100\% & 83\%  & 82\% \\
          &       &       &       & 0.5   & \texttt{TL} & \texttt{TL} & 21.99 & 100\% & 100\% & 79\%  & 77\% \\\cline{4-12}
          &       &       & \multirow{3}[0]{*}{8} & 0.1   & \texttt{TL} & \texttt{TL} & 40.89 & 100\% & 100\% & 97\%  & 94\% \\
          &       &       &       & 0.25  & \texttt{TL} & \texttt{TL} & 49.51 & 100\% & 100\% & 91\%  & 86\% \\
          &       &       &       & 0.5   & \texttt{TL} & \texttt{TL} & 40.66 & 100\% & 100\% & 94\%  & 84\% \\\hline
    \multirow{9}[0]{*}{\texttt{zj}} & \multirow{9}[0]{*}{60} & \multirow{9}[0]{*}{85} & \multirow{3}[0]{*}{2} & 0.1   & \texttt{TL} & \texttt{TL} & 13.12 & 100\% & 100\% & 49\%  & 65\% \\
          &       &       &       & 0.25  & \texttt{TL} & {5235.81} & 7.29  & 100\% & 0\%   & 51\%  & 0\% \\
          &       &       &       & 0.5   & \texttt{TL} & {9603.61} & 9.33  & 100\% & 0\%   & 5\%   & 0\% \\\cline{4-12}
          &       &       & \multirow{3}[0]{*}{5} & 0.1   & \texttt{TL} & \texttt{TL} & 25.56 & 100\% & 100\% & 96\%  & 95\% \\
          &       &       &       & 0.25  & \texttt{TL} & \texttt{TL} & 27.48 & 100\% & 100\% & 90\%  & 89\% \\
          &       &       &       & 0.5   & \texttt{TL} & \texttt{TL} & 18.32 & 100\% & 100\% & 87\%  & 86\% \\\cline{4-12}
          &       &       & \multirow{3}[0]{*}{8} & 0.1   & \texttt{TL} & \texttt{TL} & 37.85 & 100\% & 100\% & 98\%  & 96\% \\
          &       &       &       & 0.25  & \texttt{TL} & \texttt{TL} & 31.05 & 100\% & 100\% & 94\%  & 90\% \\
          &       &       &       & 0.5   & \texttt{TL} & \texttt{TL} & 20.45 & 100\% & 100\% & 91\%  & 85\% \\\hline
    \end{tabular}}%
   \caption{Computational results for the MNLCLP approaches. \label{tab:mclp}}%
\end{table}%

\begin{table}[htbp]
  \centering
   \adjustbox{scale=0.8}{\begin{tabular}{|p{1.5cm}p{0.5cm}p{0.5cm}|p{0.5cm}r|rr|r|r|}\hline
          &       &       &       &       & \multicolumn{2}{c|}{CPU Time (secs)} & \multicolumn{1}{c|}{Unsolved}       & \multicolumn{1}{c|}{GAP (\%)} \\
    {Instance} & {$|V|$} & {$|E|$} & {$\gamma$} & {$R$} & PSNLCLP  &  {PSNLCLP\_1} & {PSNLCLP} &  {PSNLCLP} \\\hline\hline
   \cline{6-9}    \multirow{9}[1]{*}{\texttt{gessler}} & \multirow{9}[1]{*}{12} & \multirow{9}[1]{*}{14} & \multirow{3}[1]{*}{0.5} & 0.1   & \multicolumn{1}{r}{\texttt{TL}} & 19.15 & 100\% & 96\% \\
          &       &       &       & 0.25  & \multicolumn{1}{r}{\texttt{TL}} & 6.28  & 100\% & 89\% \\
          &       &       &       & 0.5   & \multicolumn{1}{r}{\texttt{TL}} & 1.90  & 100\% & 75\% \\\cline{4-9}
          &       &       & \multirow{3}[0]{*}{0.75} & 0.1   & \multicolumn{1}{r}{\texttt{TL}} & 30.27 & 100\% & 98\% \\
          &       &       &       & 0.25  & \multicolumn{1}{r}{\texttt{TL}} & 10.94 & 100\% & 93\% \\
          &       &       &       & 0.5   & \multicolumn{1}{r}{\texttt{TL}} & 3.39  & 100\% & 86\% \\\cline{4-9}
          &       &       & \multirow{3}[0]{*}{1} & 0.1   & \multicolumn{1}{r}{\texttt{TL}} & 39.76 & 100\% & 97\% \\
          &       &       &       & 0.25  & \multicolumn{1}{r}{\texttt{TL}} & 13.67 & 100\% & 93\% \\
          &       &       &       & 0.5   & \multicolumn{1}{r}{\texttt{TL}} & 4.74  & 100\% & 89\% \\\hline
    \multirow{9}[0]{*}{\texttt{jilin}} & \multirow{9}[0]{*}{28} & \multirow{9}[0]{*}{34} & \multirow{3}[0]{*}{0.5} & 0.1   & \multicolumn{1}{r}{\texttt{TL}} & 26.40 & 100\% & 96\% \\
          &       &       &       & 0.25  & \multicolumn{1}{r}{\texttt{TL}} & 13.29 & 100\% & 87\% \\
          &       &       &       & 0.5   & \multicolumn{1}{r}{\texttt{TL}} & 3.44  & 100\%   & 67\% \\\cline{4-9}
          &       &       & \multirow{3}[0]{*}{0.75} & 0.1   & \multicolumn{1}{r}{\texttt{OoM}} & 54.77 & 100\% & - \\
          &       &       &       & 0.25  & \multicolumn{1}{r}{\texttt{TL}} & 20.00 & 100\% & 95\% \\
          &       &       &       & 0.5   & \multicolumn{1}{r}{\texttt{TL}} & 4.60  & 100\% & 88\% \\\cline{4-9}
          &       &       & \multirow{3}[0]{*}{1} & 0.1   &     \multicolumn{1}{r}{\texttt{OoM}}   & 78.69 &   100\%    & - \\
          &       &       &       & 0.25  & \multicolumn{1}{r}{\texttt{TL}} & 24.36 & 100\% & 97\% \\
          &       &       &       & 0.5   & \multicolumn{1}{r}{\texttt{TL}} & 7.05  & 100\% & 95\% \\\hline
    \multirow{9}[0]{*}{\texttt{richmond}} & \multirow{9}[0]{*}{48} & \multirow{9}[0]{*}{44} & \multirow{3}[0]{*}{0.5} & 0.1   &   \multicolumn{1}{r}{\texttt{TL}}     & 57.79 &  100\%     &  94\%\\
          &       &       &       & 0.25  & \multicolumn{1}{r}{\texttt{TL}} & 14.49 & 100\% & 92\% \\
          &       &       &       & 0.5   & \multicolumn{1}{r}{\texttt{TL}} & 3.90  & 100\% & 71\% \\\cline{4-9}
          &       &       & \multirow{3}[0]{*}{0.75} & 0.1   & \multicolumn{1}{r}{\texttt{OoM}} & 91.99 & 100\% & - \\
          &       &       &       & 0.25  & \multicolumn{1}{r}{\texttt{TL}} & 21.33 & 100\% & 93\% \\
          &       &       &       & 0.5   & \multicolumn{1}{r}{\texttt{TL}} & 5.62  & 100\% & 91\% \\\cline{4-9}
          &       &       & \multirow{3}[0]{*}{1} & 0.1   &   \multicolumn{1}{r}{\texttt{OoM}}    & 116.10 &   100\%    &  -\\
          &       &       &       & 0.25  & \multicolumn{1}{r}{\texttt{TL}} & 25.68 & 100\% & 94\% \\
          &       &       &       & 0.5   & \multicolumn{1}{r}{\texttt{TL}} & 7.67  & 100\% & 96\% \\\hline
    \multirow{9}[0]{*}{\texttt{foss}} & \multirow{9}[0]{*}{37} & \multirow{9}[0]{*}{58} & \multirow{3}[0]{*}{0.5} & 0.1   & \multicolumn{1}{r}{\texttt{TL}} & 41.95 & 100\% & 96\% \\
          &       &       &       & 0.25  & \multicolumn{1}{r}{\texttt{TL}} & 14.21 & 100\% & 92\% \\
          &       &       &       & 0.5   & \multicolumn{1}{r}{\texttt{TL}} & 6.83  & 100\% & 75\% \\\cline{4-9}
          &       &       & \multirow{3}[0]{*}{0.75} & 0.1   &   \multicolumn{1}{r}{\texttt{OoM}}    & 111.96 &   100\%    &  -\\
          &       &       &       & 0.25  & \multicolumn{1}{r}{\texttt{TL}} & 26.74 & 100\% & 97\% \\
          &       &       &       & 0.5   & \multicolumn{1}{r}{\texttt{TL}} & 11.54 & 100\% & 94\% \\\cline{4-9}
          &       &       & \multirow{3}[0]{*}{1} & 0.1   &   \multicolumn{1}{r}{\texttt{OoM}}    & 230.93 &   100\%    &  -\\
          &       &       &       & 0.25  &   \multicolumn{1}{r}{\texttt{OoM}}    & 61.73 &   100\%   & - \\
          &       &       &       & 0.5   &   \multicolumn{1}{r}{\texttt{OoM}}    & 19.59 &   100\%    &  -\\\hline
    \end{tabular}}%
   \caption{Computational results for the PSNLCLP approaches. \label{tab:pscp}}%
\end{table}%

\begin{figure}
\begin{center}
\includegraphics[scale=0.25]{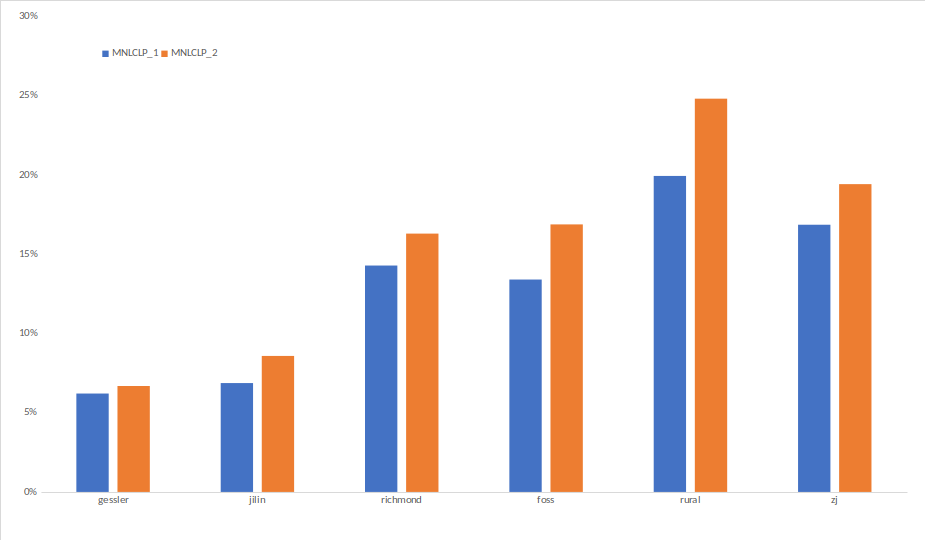}~\includegraphics[scale=0.3]{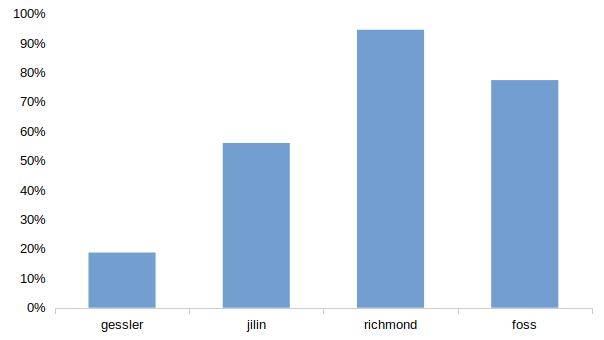}
\end{center}
\caption{Average coverage deviations of the MNLCLP\_1 and MNLCLP\_2 approach with respect to MNLCLP (left) and PSNLCLP\_1 approach for PSNLCLP (right).\label{fig:devs}}
\end{figure}

\begin{figure}
\centering
\includegraphics[width=1\linewidth]{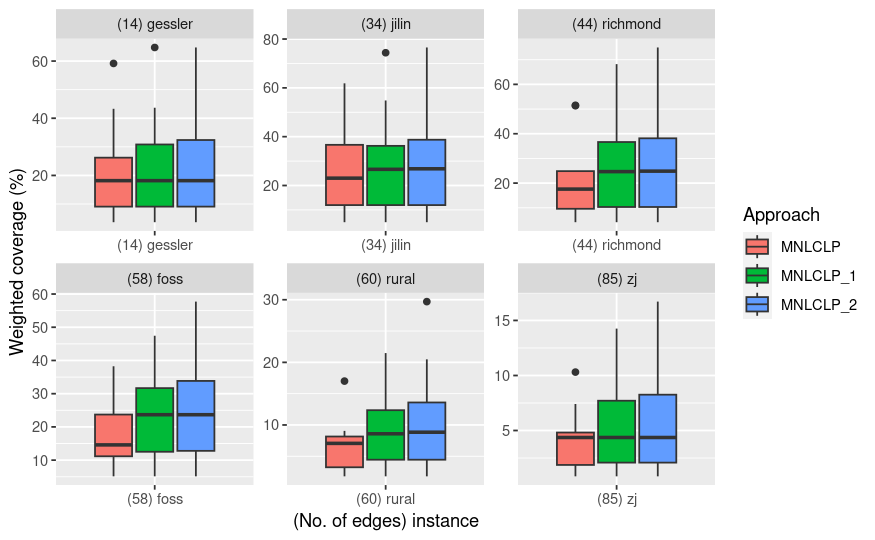}
\caption{Weighted coverage comparison between the solutions of the three approaches MNLCLP, MNLCLP\_1, and MNLCLP\_2.\label{fig:boxplot}}
\end{figure}

\subsection*{Node and Edge-restricted models}
Finally, we run some experiments to validate our proposal in terms of the coverages obtained with our approaches and those obtained by the node and edge-constrained versions of our model. In those models, the devices are allowed to be located only at the nodes or the edges of the network, respectively, instead of the whole space as in our model. We have run these restricted versions of the Maximal Coverage model, for the same instances and parameters as in the previous section.

First, in Figure \ref{fig_devs1} we plot the average deviations of the weighted volume coverages of these approaches with respect to those obtained with our model (the best solution obtained with our solution approaches). One can observe that, in some of the instances, the deviations are close to $40\%$, that is, our model obtained solutions covering $40\%$ more volume of the network than the restricted versions. As expected, the edge-restricted version covers more length of the network than the node-restricted model, although, in most of the instances, the differences are less than $5\%$. 

\begin{figure}
    \centering\includegraphics{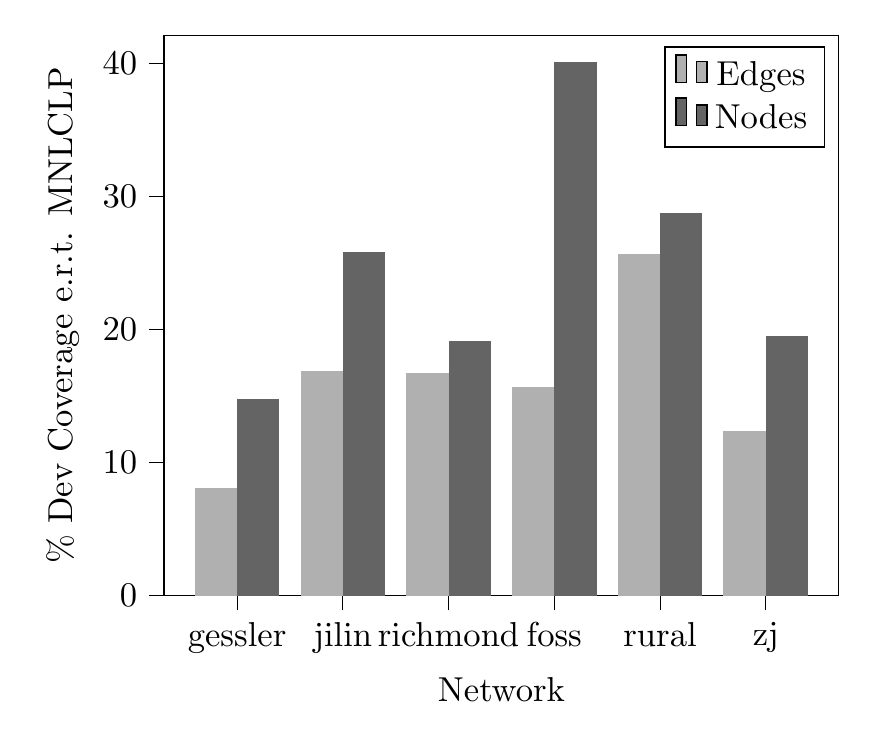}
    \caption{Average deviations of the coverages obtained with the node and edge-restricted versions of MNLCLP with respect to MNLCLP. \label{fig_devs1}}
\end{figure}

The detailed average results for the different values of $p$ and $R$ are shown in Table \ref{table:edgeandnodes}. As can be observed, as larger the number of devices to be located the larger the deviations with respect to the two restricted models. Nevertheless, the deviations do not seem to follow a clear pattern with respect to the radii. For some instances, the deviations are larger for smaller values of $R$ whereas for others the deviations are larger for the larger values.

\begin{table}[htbp]
  \centering
    \begin{tabular}{cp{0.2cm}rrr}
    {\small Network} & $p$ & $R$ & {\small Dev\_Nodes} & {\small Dev\_Edges}\\\hline\hline
    \multirow{9}[0]{*}{\texttt{gessler}} & \multirow{3}[0]{*}{2} & 0.1   & 9.94\% & 1.20\% \\
          &       & 0.25  & 9.84\% & 1.11\% \\
          &       & 0.5   & 9.82\% & 1.10\% \\\cline{3-5}
          & \multirow{3}[0]{*}{5} & 0.1   & 11.99\% & 2.78\% \\
          &       & 0.25  & 13.92\% & 6.22\% \\
          &       & 0.5   & 11.66\% & 4.51\% \\\cline{3-5}
          & \multirow{3}[0]{*}{8} & 0.1   & 23.46\% & 22.10\% \\
          &       & 0.25  & 20.47\% & 19.13\% \\
          &  &0.5 & 21.42\% & 14.82\% \\\hline
    \multirow{9}[0]{*}{\texttt{jilin}} & \multirow{3}[0]{*}{2} & 0.1   & 22.72\% & 0.97\% \\
          &       & 0.25  & 19.88\% & 1.53\% \\
          &       & 0.5   & 20.38\% & 10.49\% \\\cline{3-5}
          & \multirow{3}[0]{*}{5} & 0.1   & 24.89\% & 16.42\% \\
          &       & 0.25  & 18.26\% & 17.48\% \\
          &       & 0.5   & 21.50\% & 9.88\% \\\cline{3-5}
          & \multirow{3}[0]{*}{8} & 0.1   & 49.72\% & 45.95\% \\
          &       & 0.25  & 24.79\% & 38.00\% \\
          &       & 0.5   & 29.79\% & 11.23\% \\\hline
    \multirow{9}[0]{*}{\texttt{richmond}} & \multirow{3}[0]{*}{2} & 0.1   & 12.62\% & 2.32\% \\
          &       & 0.25  & 9.53\% & 3.53\% \\
          &       & 0.5   & 16.67\% & 9.03\% \\\cline{3-5}
          & \multirow{3}[0]{*}{5} & 0.1   & 13.98\% & 6.72\% \\
          &       & 0.25  & 19.83\% & 9.30\% \\
          &       & 0.5   & 14.48\% & 17.60\% \\\cline{3-5}
          & \multirow{3}[0]{*}{8} & 0.1   & 27.39\% & 32.84\% \\
          &       & 0.25  & 23.76\% & 38.96\% \\
          &       & 0.5   & 33.90\% & 29.89\% \\\hline

    \end{tabular}~\begin{tabular}{cp{0.2cm}rrr}
    {\small Network} & $p$ & $R$ & {\small Dev\_Nodes} & {\small Dev\_Edges}\\\hline\hline
    \multirow{9}[0]{*}{\texttt{foss}} & \multirow{3}[0]{*}{2} & 0.1   & 32.77\% & 2.14\% \\
          &       & 0.25  & 32.82\% & 2.08\% \\
          &       & 0.5   & 26.36\% & 4.32\% \\\cline{3-5}
          & \multirow{3}[0]{*}{5} & 0.1   & 32.80\% & 22.35\% \\
          &       & 0.25  & 36.47\% & 9.38\% \\
          &       & 0.5   & 19.12\% & 9.55\% \\\cline{3-5}
          & \multirow{3}[0]{*}{8} & 0.1   & 60.73\% & 40.86\% \\
          &       & 0.25  & 61.87\% & 17.85\% \\
          &       & 0.5   & 57.66\% & 32.72\% \\\hline
    \multirow{9}[0]{*}{\texttt{rural}} & \multirow{3}[0]{*}{2} & 0.1   & 19.94\% & 0.67\% \\
          &       & 0.25  & 19.42\% & 0.43\% \\
          &       & 0.5   & 6.03\% & 1.14\% \\\cline{3-5}
          & \multirow{3}[0]{*}{5} & 0.1   & 23.77\% & 22.78\% \\
          &       & 0.25  & 20.44\% & 45.94\% \\
          &       & 0.5   & 16.37\% & 39.60\% \\\cline{3-5}
          & \multirow{3}[0]{*}{8} & 0.1   & 40.18\% & 58.42\% \\
          &       & 0.25  & 80.81\% & 31.23\% \\
          &       & 0.5   & 31.41\% &  30.54\%\\\hline
\multirow{9}[0]{*}{\texttt{zj}} & \multirow{3}[0]{*}{2} & 0.1   & 1.32\% & 0.97\% \\
          &       & 0.25  & 1.31\% & 0.08\% \\
          &       & 0.5   & 5.86\% & 2.59\% \\\cline{3-5}
          & \multirow{3}[0]{*}{5} & 0.1   & 14.70\% & 9.53\% \\
          &       & 0.25  & 5.67\% & 8.58\% \\
          &       & 0.5   & 12.35\% & 8.18\% \\\cline{3-5}
          & \multirow{3}[0]{*}{8} & 0.1   & 69.01\% & 26.19\% \\
          &       & 0.25  & 33.11\% & 27.69\% \\
          &       & 0.5   & 32.38\% &  27.59\%\\\hline
\end{tabular}

  \caption{Average deviations of node and edge-restricted MNLCLP with respect to MNLCLP for the different values of $p$ and $R$.\label{table:edgeandnodes}}
\end{table}%

In Figures \ref{fig:foss} and \ref{fig:rural} we show the best solutions obtained with the three models for two of the instances (\texttt{foss} with $p=5$ and \texttt{rural} with $p=8$ both of them with $R=0.5$). In the left plot of both figures, we show the optimal solution of our model. The center and right plots are the solutions obtained with the edge-restricted and node-restricted versions of the MNLCLP, respectively. As can be observed, the optimal location of the devices differs for the different models. Specifically, the MNLCLP takes advantage of locating devices outside the edges of the network to cover edges with a high volume, whereas the restricted versions do not have such flexibility. In Figure \ref{fig:foss}, one can observe that, for the \texttt{foss} network, there is a high concentration of weighted volume at the edges in the bottom right corner of the network. Thus, the three models try to locate the devices to cover that demand. In the MNLCLP, a single device (outside the edges and nodes) suffices to cover most of the demand, whereas the restricted models require two or three devices to cover a similar amount of volume. This flexibility directly affects global coverage. For this network, our model was able to cover $10\%$ and $20\%$  more volume  than the edge and node-restricted models, respectively. The situation for the \texttt{rural} network (Figure \ref{fig:rural}) is even more impressive since our model obtained a solution with more than $30\%$ coverage than the restricted models.

\begin{figure}[h]
\centering
\includegraphics[width=0.33\linewidth]{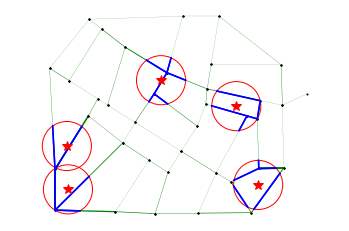}~\includegraphics[width=0.35\linewidth]{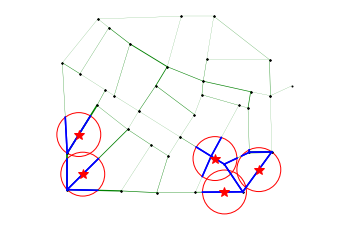}~
\includegraphics[width=0.35\linewidth]{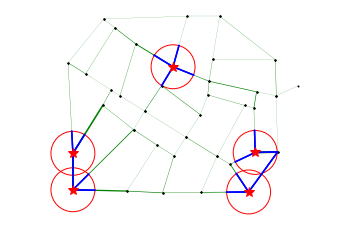}
\caption{Solutions of the MNLCLP (left), MNLCLP-edge restricted (center), and MNLCLP-node restricted (right) for instance  \texttt{foss} ($p=5$).\label{fig:foss}}
\end{figure}

\begin{figure}[h]
\centering
\includegraphics[width=0.33\linewidth]{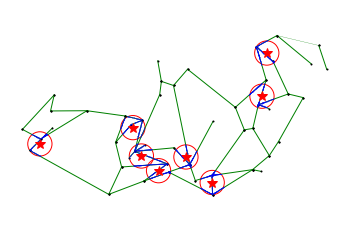}~\includegraphics[width=0.35\linewidth]{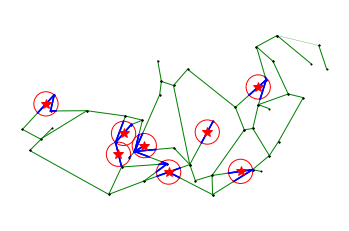}~
\includegraphics[width=0.35\linewidth]{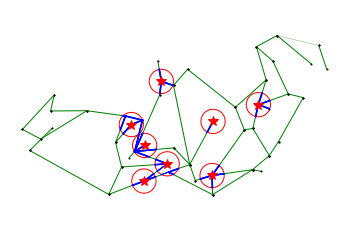}
\caption{Solutions of the MNLCLP (left), MNLCLP-edge restricted (center), and MNLCLP-node restricted (right) for instance  \texttt{rural} ($p=8$).\label{fig:rural}}
\end{figure}

One can conclude that our model is adequate to locate devices that maximize the volume coverage of a network in case they can be located at any place of the space where the network lives, as in the case of leak detection devices, since the obtained solutions significantly outperform (in coverage) the classical node and edge-restricted versions of the problem.

\section{Conclusions and Future Research}\label{sec:6}

In this paper, we study a covering location problem with direct application to the determination of optimal positions of leak detection devices in urban pipeline networks. We propose a general framework for two different versions of the problem. On the one hand, in case the number of devices is known, we derive the Maximal Network Length Covering Location Problem whose goal is to maximize the length of the network for which the device can detect the leak. On the other hand, in case the number of devices is unknown, the Partial Set Network Length Covering Location Problem aims to minimize the number of devices to locate to be able to detect the leaks in a given percent of the length of the network. We derive a method to construct initial
solutions as well as a math-heuristic algorithm. We run our models on different real-world urban water supply pipeline networks and compare the performance of the different proposals. 

Future research lines on the topic include incorporating more sophisticated coverage shapes for the devices, as non-convex shapes obtained by the union of different polyhedral and $\ell_\tau$-norm balls. It would require a further study of $\tau$-order cone constraints, as well as the representation of the union by means of disjunctive constraints, being then a challenge to provide solutions for real-world networks. In this case, it would be advisable to design efficient heuristic approaches capable to scale to large networks adequately.

\section*{Acknowledgements}

The authors of this research acknowledge financial support by the Spanish Ministerio de Ciencia y Tecnologia, Agencia Estatal de Investigacion and Fondos Europeos de Desarrollo Regional (FEDER) via project PID2020-114594GB-C21 and AEI grant number RED2022-
134149-T (Thematic Network: Location Science and Related Problems). The authors also acknowledge partial support from projects FEDER-US-1256951, Junta de Andalucía P18-FR-1422, P18-FR-2369,  B-FQM-322-UGR20, NetmeetData: Ayudas Fundación BBVA a equipos de investigación cient\'ifica
2019, and the IMAG-Maria de Maeztu grant CEX2020-001105-M /AEI /10.13039/501100011033. The first author also acknowledges the financial support of the European Union-Next GenerationEU through the program``Ayudas para la  Recualificaci\'on del Sistema Universitario Espa\~nol 2021-2023''.

\appendix

\section{Proof of Lemma \ref{lema:punto-seg}}\label{appendix}

\begin{proof}
Let $S$ be the intersection point between the line induced by $e$, $r$, and its orthogonal line passing through the point $Q$. We denote by $\mu$ the parameterization of $S$ in the ray induced by the segment pointed at $o_e$.  Thus, $\|Q-S\|=\min\{\|Q-T\|: T \in r\}$. Since $S\in r$, one can parameterize $S$ as $S= (1-\mu)o_e + \mu f_e$ for some $\mu \in \R$. Let us analyze the different possible values for $\mu$:
\begin{itemize}
\item If $\mu\in [0,1]$, one gets that: 
\begin{align*}
\|Q-(\mu(f_e-o_e)+o_e)\|=\|Q-S\|&=\min\{\|Q-T\|: T \in r\}\\
&\leq \min\{\|Q-T\|: T \in e\}=\delta(e,Q).
\end{align*}
\item If $\mu<0$, will show that $\delta(e,Q)=\|Q-o_e\|$. Let $\lambda\in [0,1]$ and $X=(1-\lambda) o_e + \lambda f_e \in e$. Then:	 \begin{align*}
	 \|Q-o_e\|^2 & = \|Q-S\|^2 + \|S-o_e\|^2  =\|Q-S\|^2 + \|\mu (f_e-o_e)+o_e-o_e\|^2\\
	 & =\|Q-S\|^2 + |\mu|^2\|(f_e-o_e)\|^2  \leq \|Q-S\|^2 + |(\mu-\lambda)|^2\|(f_e-o_e)\|^2\\
	 & =\|Q-S\|^2 + \|\mu(f_e-o_e)+o_e-(\lambda(f_e-o_e)+o_e)\|^2 =\|Q-S\|^2 + \|S-X\|^2\\
	 &=\|Q-X\|^2.
	 \end{align*}
\item In case $\mu>1$, let us see that $\delta(e,Q)=\|Q-f_e\|$. Let $\lambda\in [0,1]$ and $X=\lambda (f_e-o_e)+o_e$ be in $e$:
	 \begin{align*}
	 \|Q-f_e\|^2 & = \|Q-S\|^2 + \|S-f_e\|^2 =\|Q-S\|^2 + \|\mu (f_e-o_e)+o_e-f_e\|^2\\
	 & =\|Q-S\|^2 + \|\mu (f_e-o_e)+o_e-f_e+o_e-o_e\|^2\\
	 & =\|Q-S\|^2 + |\mu-1|^2\|(f_e-o_e)\|^2\\
	 & \leq \|Q-S\|^2 + |(\mu-\lambda)|^2\|(f_e-o_e)\|^2\\
	 & =\|Q-S\|^2 + \|\mu(f_e-o_e)+o_e-(\lambda(f_e-o_e)+o_e)\|^2\\
	 & =\|Q-S\|^2 + \|S-X\|^2\\
	 &=\|Q-X\|^2.
	 \end{align*}
	 \end{itemize}
	 
	Summarizing, we get that the point in $e$ closest to $Q$ is in the form $(1-\lambda)o_e + \lambda f_e$ with $$\lambda=\begin{cases}
	 0 & \mbox{if $\mu<0$,}\\
	 \mu & \mbox{if $0\leq \mu \leq 1$,}\\
	 1 & \mbox{if $\mu>1$.}
	 \end{cases}$$
	 that is, $\lambda = \min\{\max\{0,\mu\}, 1\}$, being then $\delta(e,Q) =\|Q -  (\min\{\max\{0,\mu\}, 1\} (f_e-o_e) + o_e)\|.$
\end{proof}

\end{document}